\documentclass{amsart}
\usepackage{amsfonts}
\usepackage{amssymb}
\usepackage{mathtools}
\mathtoolsset{showonlyrefs} %	enumera solo las ecuaciones que se citan
\usepackage{nicefrac}
\usepackage{enumerate}
\usepackage{url}
\usepackage[dvipsnames]{xcolor}      % Need the color package
\usepackage{epsfig}
\usepackage{graphicx}
\usepackage{tikz}
\usepackage{forest}
\usepackage{esint}
\usepackage{enumitem}
\usepackage{comment}
\usepackage{appendix}
\usepackage{mathrsfs}
\usepackage{hyperref}
\hypersetup{
colorlinks = true,
linkcolor  = red,      % \ref, \eqref, \label -> solo el numero en rojo, sin caja
citecolor  = black,    % \cite (dejalo como quieras)
urlcolor   = black,    % \url, \href (dejalo como quieras)
pdfborder  = {0 0 0}   % por las dudas, fuerza a no dibujar caja en ningun visor
}

\newtheorem{teo}{Theorem}[section]
\newtheorem{lem}[teo]{Lemma}
\newtheorem{co}[teo]{Corollary}
\newtheorem{pro}[teo]{Proposition}

\theoremstyle{remark}
\newtheorem{re}[teo]{Remark}

\theoremstyle{definition}
\newtheorem{de}[teo]{Definition}

\def\R{{\mathbb R}}

\def\N{{\mathbb N}}

\def\T{{\mathbb{T}_m}}
\def\A{\mathcal{A}_{\beta}}
\def\S{\mathcal{S}}
\def\p{p_\beta}

\def\B{{\mathscr{B}}}

\title[The evolution equation and the eigenvalue problem on a regular tree]{The evolution equation and the eigenvalue problem for the Laplacian on a regular tree}
\author[L. M. Del Pezzo, N. Frevenza and J. D. Rossi]
{Leandro M. Del Pezzo, Nicol{\'a}s Frevenza and Julio D. Rossi}

\address{Leandro M. Del Pezzo, Nicol{\'a}s Frevenza
\hfill\break\indent
Instituto de Estad{\'i}stica, FCEA,
Universidad de la Rep{\'u}blica and PEDECIBA,
\hfill\break\indent Gonzalo Ram{\'i}rez 1926 (11200),
Montevideo, Uruguay.}

\address{Julio D. Rossi
\hfill\break\indent
Departamento de Matem{\'a}tica y Estadistica,  Universidad Torcuato Di Tella.
\hfill\break\indent  Av. Figueroa Alcorta 7350, Buenos Aires, Argentina.
}

\email{{\tt leandro.delpezzo@fcea.edu.uy, nicolas.frevenza@fcea.edu.uy,
julio.rossi@utdt.edu}}

\begin{document}

\begin{abstract}
	In this paper our main goal is to study the evolution problem
	associated with the Laplacian operator with Dirichlet boundary
	conditions in a regular tree.
	 
	First, we prove existence and uniqueness of globally
	bounded solutions for every bounded initial datum.	
	 
	Next, we deal with the asymptotic behavior of the solutions and
	we prove that they decay to zero exponentially fast.
	This decay rate is governed by the associated first eigenvalue
	that we also study here.
\end{abstract}

\maketitle

%%%%%%%%%%%%%%%%%%%%%%%%%%%%%%%%%%%%%%%%%%%%%%%%%%%%%%%%%%%%%%%%%%%%%%%%%%%%%%%%%%%%%%%%%%%%%%%
\section{Introduction}
%%%%%%%%%%%%%%%%%%%%%%%%%%%%%%%%%%%%%%%%%%%%%%%%%%%%%%%%%%%%%%%%%%%%%%%%%%%%%%%%%%%%%%%%%%%%%%%

    Our main goal in this paper is to analyze the asymptotic
    behavior of solutions to an evolution problem
    analogous to the heat equation with homogeneous
    Dirichlet boundary conditions when the ambient
    space is a regular \(m-\)branching tree \(\T\)
    (a connected infinite graph without cycles where each
    vertex has exactly \(m\) successors).

    Given an initial datum \(f\colon \T \to \R\), we look for
    \(u\colon \T\times [0,\infty)\to \R\) satisfying
    \begin{equation}
	    \label{ec.evolucion}
	    \begin{cases}
		    u_t(x,t) - \Delta_{\beta}u(x,t) = 0 & \text{ in } \T
		    \times(0,\infty), \\
		    u(x,0) = f(x) & \text{ in } \T, \\
		    \lim\limits_{x\to y} u(x,t)=0
		    & \text{ for every } y\in\partial\T,
	    \end{cases}
    \end{equation}
 interpreted as taking limits along branches of the tree,
uniformly with respect to the branch, for each fixed \(t>0\).    
    
    Here \(\Delta_{\beta}u (x,t)\) involves a mean value
    (with weights depending on a parameter \(\beta\)) of \(u\) among the
    values of \(u (\cdot,t)\) at vertices connected to \(x\)
    (see the precise definition below).
    The Dirichlet boundary condition
    \(\lim\limits_{x\to y} u(x,t)=0\) has to be
    interpreted as taking limits along branches of the tree.

    Regular trees and mean value averaging operators on them serve as a discrete model analogous 
    to the unit ball together with  the continuous partial differential equations defined on it.
    Therefore, we expect
    \eqref{ec.evolucion}
    to be an analogous problem to the heat equation with Dirichlet boundary conditions, however,
    significant differences appear.

	 For this problem, we prove existence and uniqueness of globally
	bounded solutions for every bounded initial datum, and then
	study their asymptotic behavior as \(t\to+\infty\). 
    As happens for solutions to the classical heat
    equation in a bounded domain of the Euclidean space, we find that also in our case
    solutions decay to zero uniformly in \( \T\) as \(t\to + \infty\) with an exponential rate.
    This exponential rate is given by the first eigenvalue of the operator (here we also analyze
    the principal eigenvalue for \(\Delta_{\beta}u\)). The analysis of the principal eigenvalue is delicate since
    very different scenarios appear depending on the parameter \(\beta\) (for a certain range of \(\beta\) there
    is no principal eigenvalue, for a different range there is a principal eigenvalue and for \(\beta=0\) every
    \(\lambda \in (0,1]\) is a principal eigenvalue).

	Related eigenvalue problems for birth--death processes
	on trees were studied by Wang and Zhang~\cite{WZ2013},
	who obtained variational formulas for the first
	Dirichlet eigenvalue with the root as an absorbing
	boundary. In our problem, the root remains part of
	the evolution, while the Dirichlet condition is imposed
	at infinity. Spectral properties of weakly spherically
	symmetric weighted graphs were investigated by Keller,
	Lenz, and Wojciechowski~\cite{KellerLW2013}.
	Here we use comparison principles, radial recurrences,
	and finite truncations to study the eigenvalue and
	evolution problems for \(\Delta_\beta\).

    \subsection{Setting and statements of the main results}
    Let us first introduce some notations needed for the precise statement of our
    results.

    Let \(m\in\mathbb{N}_{\ge2}\).
    A regular \(m-\)branching tree \(\T\) is a connected graph without cycles where each
    vertex has \(m\) successors.
    Formally, \(\T\) consists of a root (denoted by \(\emptyset\)) and all finite sequences
    \((\emptyset,a_1,a_2,\dots,a_k)\) with \(k\in\N\),
    whose coordinates \(a_i\) are chosen from \(\{0,1,\dots,m-1\}\).
    A successor of vertex \(x\) is obtained by adding another coordinate to \(x\),
    that is \((x,i)\) with \(i\in \{0,1,\dots,m-1\}\).

    Notice that when \(x\neq\emptyset\), \(x\) has
    only an immediate predecessor
    (or ancestor) which is denoted by \(\hat{x}\).
    We say that \(x\in\T\) has level \(k\) when \(x=(\emptyset,a_1,a_2,\dots,a_k)\)
    for some  \(k\in\mathbb{N}\). We write \(|x|\) for the level of \(x\).

    An infinite sequence of vertices starting at the root \(\emptyset\) such that each
    vertex is followed by an immediate successor, is called a branch of \(\T\).
    The collection of all branches is the boundary of \(\T\),
    and we write \(\partial\T\) for it.
    Observe that the map \(\psi:\partial\T\to[0,1]\) defined as
    \[
	    \psi(x)
	    \coloneqq
	    \sum_{k=1}^{+\infty} \frac{a_k}{m^{k}}
    \]
    is surjective, where \(x=(\emptyset,a_1,\dots, a_k,\dots)\in\partial\T\) and
    \(a_k\in\{0,1,\dots,m-1\}\) for all \(k\in\mathbb{N}\).
    Whenever \(x=(\emptyset,a_1,\dots,a_k)\) is a vertex,
    the map \(\psi\) is extended by setting
    \[
	    \psi(x)\coloneqq\psi(\emptyset,a_1,\dots,a_k,0,\dots,0,\dots).
    \]
    We write also \(\psi\) for this extension.

    In this context, we study the evolution problem and the eigenvalue problem associated to a certain type of
    Laplace operator on the tree. For \(u\colon\T\to\mathbb{R}\) a function, and
    \(\beta\in[0,1),\) the \(\beta\)-Laplace operator is defined as
    \[
	    \Delta_\beta u(x)
	    \coloneqq
	    \begin{cases}
		    \displaystyle \frac{1}{m} \sum_{i=0}^{m-1} u(\emptyset,i)-u(\emptyset)
		    &\text{if } x=\emptyset,\\
		    \left(\beta u(\hat{x})+\displaystyle \frac{1-\beta}{m} \sum_{i=0}^{m-1}
		    u(x,i)-u(x)\right)\p^{-|x|}
		    &\text{if } x\in\T\setminus\{\emptyset\},
	    \end{cases}
    \]
    where
    \[
	    \p\coloneqq
	    \begin{cases}
		    1 &\text{if } \beta=0,\\
		    \frac{\beta}{1-\beta} &\text{if } \beta\in(0,1).
	    \end{cases}
    \]
    It is a weighted version of the Laplacian in \(\T\).

   Note that if \(\beta=\frac{1}{m+1}\), the functions that are harmonic
   with respect to the \(\beta\)-Laplace operator are precisely those
   that are harmonic with respect to the classical graph Laplacian
   on the tree. In this case, harmonicity at a node involves the values
   of \(u\) at that node and all its neighbors. 
   For \(\beta=0\), the notion of harmonicity coincides with that
   on the rooted tree viewed as a directed graph, with edges oriented
   from each node to its successors. In this case, the operator
   involves only the values of \(u\) at a node and its successors.
   These types of operators are well known in the literature.
   See, for instance, \cite{anandam,DPMR1,DPMR2,DPMR3,KLW,KW}
   and the references therein.

    In \cite{DPFR1}, we show that for a datum \(g\in C([0,1],\R)\)
    there exists a unique solution of the Dirichlet problem
    \[
	    \begin{cases}
		    -\Delta_\beta u =0 &\text{on } \T,\\
		    \lim\limits_{x\to y} u(x)=g(y) & \text{for } y\in \partial\T,
	    \end{cases}
    \]
    when \(\beta<\tfrac12\), while when \(\beta \geq \tfrac12\)
    there are no non-constant bounded solutions.
    Hence the Dirichlet problem with a continuous datum
    \(g\) is solvable only when \(\beta <\tfrac12\)
    or when \(g\) is a constant function.
    
    Boundary limits are taken along the vertices of each branch.
    Unless otherwise stated, convergence in a boundary condition
    is understood to be uniform with respect to the branch.
    In particular, the homogeneous boundary condition for a
    function \(h\colon\T\to\R\) means
    \[
  		\sup_{|x|\ge N}|h(x)|\xrightarrow[N\to\infty]{}\ 0.
  	\]
    For a time-dependent function, this condition is imposed
    at each fixed \(t>0\). Any additional uniformity with
    respect to time will be stated explicitly. 

    In this article, we study the eigenvalue problem for \(\Delta_\beta\)
    with homogeneous Dirichlet boundary condition.
    We say that \(\lambda\in\mathbb{R}\) is an eigenvalue of \(\Delta_\beta\)
    if there is a bounded function \(u\colon\T\to\mathbb{R}\) such that \(u\not\equiv 0\) and
    \begin{equation}
	    \label{prob.autovalores.def}
	    \begin{cases}
		    -\Delta_\beta u =\lambda u &\text{on } \T,\\
		    \lim\limits_{x\to y} u(x)=0 & \text{for } y\in \partial\T.
	    \end{cases}
    \end{equation}
    In such case, \(u\) is called an eigenfunction associated to \(\lambda.\)
    An eigenvalue \(\lambda>0\) for \(\Delta_\beta\) is called a principal eigenvalue if there is a
    non-negative eigenfunction \(u\) associated to \(\lambda\).
    Our aim is to prove that a principal eigenvalue exists for some values of \(\beta\),
    to provide a characterization of
    this eigenvalue and find bounds for it depending on \(\beta\).

    We begin by considering the case \(\beta=0.\)

    \begin{teo}\label{Intro casobeta=0}
	    Let \(\beta=0\). Then \(\lambda\) is a principal
	    eigenvalue of \(\Delta_\beta\)
	    if  and only if \( \lambda\in (0,1].\) Furthermore, if
	    \(\lambda \in (0,1],\) then
	    \[
		    u(x)=(1-\lambda)^{|x|}
	    \]
	    is an eigenfunction associated to \(\lambda\).
	    %Finally if \( \lambda\in (0,1),\)
	    %then \(\lambda\) is simple.
    \end{teo}

    It is easy to check that
    \[
	    u(\emptyset)= 1 \text{ and } u(x)=0 \, \text{ for all }
	    x\neq \emptyset
    \]
    and
    \[
	    v(\emptyset)=v(\emptyset,0)=1, \, v(\emptyset,1)=-1
	    \text{ and } v(x)=0 \, \text{ in other case }
    \]
    are two different eigenfunctions associated to \(\lambda=1.\)
    Therefore, \(\lambda=1\) also has associated eigenfunctions
    that change sign.

    \medskip

    Next, we deal with the case \(0 < \beta < \frac12\).

    \begin{teo} \label{primer_teorema}
	    Let \(\beta\in (0,\tfrac12).\)
	    Define \(\lambda_1(\beta)\coloneqq\sup\A\) where
	    \[
		    \A
		    \coloneqq
		    \left\{\lambda >0 \colon
		    \exists v \colon \T\to\R \text{ with }
		    0< c < v < C, \text{ satisfying } \Delta_{\beta}v + \lambda v \le 0 \text{ in } \T
		    \right\}.
	    \]
	    Then, \(\lambda_1(\beta)\) is the unique principal eigenvalue of \(\Delta_\beta\).

	    Moreover,
	    %$\lambda_1(\beta)$ is simple and
	    there exists a strictly positive
	    eigenfunction \(\mathrm{u}\colon \T \to \R\) such that
	    \(\mathrm{u}\) is constant at each level of \(\T\) and
	    it is decreasing with respect to the levels,
	    that is \(\mathrm{u}(x)=f(|x|)\) where
	    \(f\colon\mathbb{N}_0\to\mathbb{R}\) is decreasing.
    \end{teo}

    In the following proposition we find some bounds for the value of \(\lambda_1\) in terms of
    \(\beta\).

    \begin{pro}
	    \label{intro_prop_cota}
	    Let \(\beta\in(0,\tfrac12).\)
	    Then \(\lambda_1(\beta)\) is a non-increasing function in \((0,\frac12),\) and
	    \[
		    \frac{(1-2\beta)^2}{\beta^2 + (1-\beta)^2}\le
		    \lambda_1(\beta)\le \frac{1-2\beta}{1-\beta}
		    \quad\forall \beta\in(0,\tfrac12).
	    \]
    \end{pro}

    In the case \(\beta\in(\frac12,1),\) no principal eigenvalue exists.

    \begin{teo}\label{intro_casobeta>12}
	    Let \(\beta\in [\frac12,1)\). If \(\lambda\) is an
	    eigenvalue of \(\Delta_\beta\)
	    then \( \lambda\) is not principal.
    \end{teo}

    Once we have clarified the eigenvalue problem, our next goal is to study the  evolution equation
    associated with \(\Delta_{\beta}.\)
    Recall from the Introduction the problem of finding
    \(u\colon \T\times [0,\infty)\to \R\),
    for a given initial datum \(f\colon \T \to \R\),
    satisfying \eqref{ec.evolucion}.
    It should be highlighted that, in the case \(\beta=0\), this evolution equation was first
    studied in \cite{DPMR3}.

    The present paper goes beyond \cite{DPMR3} in three ways. First, we treat the whole family
    of operators \(\Delta_\beta\), \(\beta\in[0,1)\), rather than only the unweighted case
    \(\beta=0\); working with a general \(\beta\) forces the comparison and barrier arguments to
    track a vertex-dependent weight \(\p^{-|x|}\) throughout, instead of a constant one. Second, we
    give a complete description of the principal eigenvalue across this whole range of \(\beta\):
    every \(\lambda\in(0,1]\) is principal when \(\beta=0\) (Theorem~\ref{Intro casobeta=0}), there
    is a unique principal eigenvalue \(\lambda_1(\beta)\) when \(\beta\in(0,\tfrac12)\)
    (Theorem~\ref{primer_teorema}), and none exists when \(\beta\in[\tfrac12,1)\)
    (Theorem~\ref{intro_casobeta>12}); this trichotomy, and in particular the non-existence result,
    has no counterpart in \cite{DPMR3}. Third, we show that the decay rate \(e^{-\lambda_1(\beta)t}\)
    is optimal, by exhibiting data for which it is attained exactly
    (Theorem~\ref{intro_teo_evo_2}), tying the asymptotics of the evolution problem to the
    eigenvalue \(\lambda_1(\beta)\) constructed here.

    \begin{teo}\label{intro_teo_exist_1}
	    Let \(\beta\in (0,\tfrac12),\) and \(f\in L^{\infty}(\T)\).
	    There exists a unique solution
	    \(u\in L^{\infty}(\T\times[0,\infty))\) to \eqref{ec.evolucion}.
    \end{teo}

	{    
     For initial data that vanish uniformly at the boundary,
the solution also satisfies a stronger boundary property,
uniformly down to the initial time.

    We isolate this hypothesis on the datum, since it recurs throughout the paper:
    we say that \(f\in L^\infty(\T)\) satisfies condition \textup{(L)} if
    \begin{equation}
	    \tag{L}\label{cond:L}
	    \sup_{|x|\ge N}|f(x)|\ \xrightarrow[N\to\infty]{}\ 0,
    \end{equation}
    that is, if \(f\) goes to zero at \(\partial\T\) uniformly.

    \begin{teo}\label{intro_teo_exist_2}
	    Let \(\beta\in (0,\tfrac12)\) and let \(f\in L^{\infty}(\T)\) satisfy condition \eqref{cond:L}.
	     Then the solution given by
Theorem~\ref{intro_teo_exist_1} satisfies
\[
    \sup_{\substack{|x|\ge N\\t\ge0}}
    |u_f(x,t)|\xrightarrow[N\to\infty]{}0.
\]

    \end{teo}

    %	The precise statement, together with the reduction of this uniqueness question to a single
    %	comparison lemma, is given in Subsection~\ref{subsec:phi_exist_unique}.
    }

    Next, we deal with the asymptotic behavior of the solution.

		\begin{teo}\label{intro_teo_evo_1}
		    Let \(\beta\in(0,\tfrac12)\) and
		    \(f\in L^\infty(\T)\), and let \(u\) be the solution
		    of \eqref{ec.evolucion} given by
		    Theorem~\ref{intro_teo_exist_1}.
		    Then, for every \(\lambda\in\A\), there exists
		    a constant \(C=C(\beta,\lambda)>0\), independent
		    of \(f\), such that
		    \[
		        |u(x,t)|
		        \le C\|f\|_\infty e^{-\lambda t}
		        \qquad
		        \text{for every }(x,t)\in\T\times(0,\infty).
		    \]
		\end{teo}

    Under an additional condition on the initial condition, \(f\), we can show that the optimal decay rate is
    \(e^{-\lambda_1(\beta)t}.\)

    \begin{teo}
	    \label{intro_teo_evo_2}
	    Let \(\beta\in(0,\tfrac12)\), and \(f\in L^{\infty}(\T)\) such that there exists \(K>0\) with
	    \[
		    |f(x)|\le K\,v(x) \quad \forall x\in \T,
	    \]
	    where \(v\) is a positive radial eigenfunction associated with \(\lambda_1(\beta)\), whose existence is established in Theorem~\ref{primer_teorema}, normalized by
\(v(\emptyset)=1\). 
	    Then
	    there is a solution \(u_f\) of \eqref{ec.evolucion} (the unique one vanishing at
	    \(\partial\T\) uniformly on compacts of \([0,\infty)\))
	    satisfying
	    \[
		    |u_f(x,t)|\le K\,e^{-\lambda_1(\beta)t} \quad \forall (x,t)\in\T\times(0,\infty).
	    \]
    \end{teo}

%%%%%%%%%%%%%%%%%%%%%%%%%%%%%%%%%%%%%%%%%%%%%%%%%%%%%%%%%%%%%%%%%%%%%%%%

\subsection*{The paper is organized as follows}

Sections~\ref{cp} and~\ref{sppe} establish comparison principles
and basic properties of eigenvalues.
Sections~\ref{cpe1} and~\ref{cpe2} treat the different ranges of
\(\beta\), proving Theorem~\ref{primer_teorema} and
Proposition~\ref{intro_prop_cota} in the latter case.
Section~\ref{eq} proves existence, uniqueness, and the uniform
boundary behavior of solutions.
Section~\ref{sec:decay} establishes their exponential decay.

%%%%%%%%%%%%%%%%%%%%%%%%%%%%%%%%%%%%%%%%%%%%%%%%%%%%%%%%%%%%%%%%%%%%%%%%%%%%%%%%%%%%%%
\section{Comparison Principles}
\label{cp}
%%%%%%%%%%%%%%%%%%%%%%%%%%%%%%%%%%%%%%%%%%%%%%%%%%%%%%%%%%%%%%%%%%%%%%%%%%%%%%%%%%%%%%
    In this section, we establish comparison principles for both elliptic and
    parabolic problems related to \(\Delta_\beta u\).
    We begin by proving a maximum principle.

    \begin{teo}[Maximum Principle]
	    \label{teo_mp}
	    Let \(\beta \in [0, 1)\) and \(u\colon\T\to\R\) be a function such
	    that
	    \begin{align}
		    -\Delta_\beta u (x) &\ge 0 \quad \text{in } \T, \label{eq-mp1}\\
		    \liminf_{x\to y} u(x) &\ge 0 \quad \text{for } y \in \partial\T.
		    \label{eq-mp2}
	    \end{align}
	    Then,
	    \[
		    u(x) \geq 0 \qquad \mbox{ for every } x \in \T.
	    \]
	    Moreover, if \(\beta > 0\), then either \(u(x) > 0\) in \(\T\) or \(u \equiv 0\)
	    in \(\T\).
    \end{teo}

    \begin{proof}
	    The case \(\beta = 0\) follows directly from \cite[Lemma 3.3]{DPMR2}.

	    Assume now that \(\beta > 0\).

	    \noindent{\it Step 1: \(u\) is non-negative}.
	    Assume, for the sake of contradiction,  that there exists \(x\in\T\) such that
	    \[
		    u(x)<0.
	    \]
	    Define
	    \[
		    k_0\coloneqq \min\left\{k\in\N_0\colon \exists x\in\T
		    \text{ such that } |x|=k \text{ and } u(x)<0\right\},
	    \]
	    Choose \(x_0\in\T\) such that \(|x_0|=k_0\) and \(u(x_0)<0.\)
	    We analyze two cases:
	    \begin{itemize}
		    \item If \(x_0=\emptyset\): Using
		    \[
			    0\le -\Delta_\beta u(\emptyset)
			    =u(\emptyset)-\frac1m\sum_{i=0}^{m-1}
			    u(\emptyset,i)
		    \]
		    and \(u(\emptyset)<0,\)
		    there exists \(i\in\{0,1,\dots,m-1\}\) such that
		    \[
			    u(\emptyset,i)\le u(\emptyset)< 0.
		    \]
		    \item If \(x_0\neq\emptyset\): Using
		    \[
			    0\le -\Delta_\beta u(x_0)
			    =\left(\beta(u(x_0)-u(\hat{x}_0))
			    +(1-\beta)\left(u(x_0)-\frac1m\sum_{i=0}^{m-1} u(x_0,i)
			    \right)\right)
			    p_\beta^{-|x_0|},
		    \]
		    and since it holds that \(u(\hat{x}_0)\ge 0\) and \(u(x_0)<0,\)
		    there exists
		    \(i\in\{0,1,\dots,m-1\}\) such that
		    \[
			    u(x_0,i)\le u(x_0)< 0.
		    \]
	    \end{itemize}

	    In either case, there exists a successor \(x_1\)  of \(x_0\) such that
	    \[
		    u(x_1)\le u(x_0)< 0.
	    \]

	    Now, using
	    \[
		    0\le -\Delta_\beta u(x_1)
		    =\left(\beta(u(x_1)-u(x_0))
		    +(1-\beta)\left(u(x_1)-\frac1m\sum_{i=0}^{m-1} u(x_1,i)\right)
		    \right)p_\beta^{-|x_1|},
	    \]
	    and that \(u(x_1)\le u(x_0)< 0,\) we obtain that
	    there is a successor \(x_2\) of \(x_1\) such that
	    \[
		    u(x_2)\le u(x_1)< 0.
	    \]

	    Repeating this argument inductively, we construct a sequence
	    \(\{x_n\}_{n\ge0}\) in  \(\T\)
	    such that \(x_{n+1}\) is a
	    successor of \(x_n\) and
	    \[
		    u(x_{n+1})\le u(x_{n})< 0
	    \]
	    for any \(n\in\N_0.\)

	    Then \(x_n\to y\in \partial\T\) and
	    \[
		    \lim_{n\to\infty} u(x_n)<0
	    \]
	    which contradicts \eqref{eq-mp2}. Thus, we have proved that \(u(x)\ge 0\) in \(\T.\)

	    \noindent{\it Step 2: Reaching \(u=0.\)}
	    Assume that
	    \(0\) is attained at \(z \in \T\), \emph{i.e.}, \(u(z) = 0\).
	    Using \eqref{eq-mp1}, we have:
	    \[
		    0 \le -\Delta_\beta u(z) =
		    -\left( \beta (u(\hat{z}) - u(z)) +
		    \frac{1-\beta}{m} \sum_{i=0}^{m-1} (u(z,i) - u(z)) \right) \p^{-|z|}
		    \le 0,
	    \]
	    if \(z \neq \emptyset\), and
	    \[
		    0 \le -\Delta_\beta u(\emptyset) =
		    -\left(
		    \frac{1}{m} \sum_{i=0}^{m-1} u(\emptyset,i) - u(\emptyset)
		    \right) \le 0,
	    \]
	    if \(z = \emptyset\).
	    In either case, \(\Delta_\beta u(z) = 0\). Since \(u(z) = 0\) and
	    \(\beta \neq 0\), it follows that \(u(x) = 0\) for all \(x \in \T\).
	    Consequently,
	    \(u(x)\) is constant, so \(u(x) > 0\) in \(\T\) or \(u(x) \equiv 0\)
	    in \(\T\).
    \end{proof}

    As a direct consequence of the Maximum Principle, we obtain the following
    comparison principle for the elliptic case. For similar results, see \cite{ary,DPFR1,DPMR2}.

    \begin{teo}[Comparison Principle]
	    \label{teo_cp_e}
	    Let \(\beta \in [0, \tfrac{1}{2})\) and \(u,v \colon \T \to \R\) be bounded
	    functions such that
	    \begin{align}
		    -\Delta_\beta u (x) &\ge -\Delta_\beta v (x) \quad \text{in } \T,
		    \label{eq:cp1}\\
		    \liminf_{x \to y} u(x) &\ge \limsup_{x \to y} v(x) \quad \text{for } y
		    \in \partial\T. \label{eq:cp2}
	    \end{align}
	    Then,
	    \[
		    u(x) \ge v(x) \quad \text{in } \T.
	    \]
	    Moreover, if \(\beta > 0\), then either \(u(x) > v(x)\) in \(\T\) or
	    \(u(x) \equiv v(x)\) in \(\T\).
    \end{teo}

    Before stating the comparison principle for the parabolic problem, we
    reformulate the evolution equation \eqref{ec.evolucion}.
    Note that \(u \in L^\infty(\T \times [0, \infty))\) solves
    \eqref{ec.evolucion} if and only if
    \[
	    u(x,t) = \mathcal{K}_\beta^f u(x,t),
    \]
    where
    \begin{equation}
	    \label{Kbeta}
	    \mathcal{K}_\beta^f u(x,t) \coloneqq e^{-t\p^{-|x|}}f(x) +
	    \int_0^t e^{(s-t)\p^{-|x|}} \left( \Delta_\beta u(x,s) +
	    \p^{-|x|}u(x,s) \right) \, ds,
    \end{equation}
  and,  for every fixed \(t>0\),
	\[
	    \sup_{|x|\ge N}|u(x,t)|\xrightarrow[N\to\infty]{}\ 0.
	\]

 Indeed, any bounded measurable fixed point of
\(\mathcal K_\beta^f\) is continuous in time at each vertex,
by the integral formula. Since every vertex has finitely
many neighbors, the integrand is then continuous in time.
Differentiation yields \(u_t=\Delta_\beta u\) for \(t>0\),
while evaluation at \(t=0\) gives \(u(x,0)=f(x)\).

%%%%%%%%%%%%%%%%%%%%%%%%%%%%%%%%%%%%%%%%%%%%%%%%%%%%%%%%%%%%%%%%%%%%%%%%%%%%%%%%%%%%%%%%%%%%%%%%%%%%%%
    
    \begin{teo}[Parabolic Comparison Principle]
	    \label{teo_cp_p}
	    Let \(\beta \in [0, \tfrac{1}{2})\) and \(f,g \in L^\infty(\T)\). If
	    \[
		    f(x) \ge g(x) \quad \text{in } \T,
	    \]
	    \(u,v \in L^\infty(\T \times [0, \infty))\), with \(t\mapsto u(x,t)\)
	    and \(t\mapsto v(x,t)\) continuous for every \(x\in\T\), satisfy
	    \[
		    u(x,t) \ge \mathcal{K}_\beta^f u(x,t) \quad \text{and} \quad v(x,t) \le
		    \mathcal{K}_\beta^g v(x,t) \quad \text{in } \T \times [0, \infty),
	    \]
		
		and, for every \(T>0\),
		\begin{equation}
		    \label{eq:parabolic_boundary_order}
		    \lim_{N\to\infty}
		    \inf_{\substack{{|x|\ge N}\\0\le t\le T}}
		    \bigl(u(x,t)-v(x,t)\bigr)\ge0,
		\end{equation}
		
		then
	    \[
		    u(x,t) \ge v(x,t) \quad \text{in } \T \times [0, \infty).
	    \]
    \end{teo}

	Condition~\eqref{eq:parabolic_boundary_order} means that,
	for every \(T>0\) and every \(\varepsilon>0\), there exists
	\(N\) such that
	\[
	    u(x,t)-v(x,t)\ge-\varepsilon
	    \qquad\text{whenever } |x|\ge N,\quad 0\le t\le T.
	\]
	The same \(N\) must work for all vertices and all times
	in the indicated range.
	
	\begin{proof}
    Set \(w=u-v\), and fix \(T>0\).
    Suppose, for contradiction, that
    \[
        -\eta\coloneqq
        \inf_{\T\times[0,T]}w<0.
    \]
    Since \(w\) is bounded, \(0<\eta<\infty\).

    By \eqref{eq:parabolic_boundary_order}, there exists
    \(N\ge1\) such that
    \[
        w(x,t)\ge-\frac{\eta}{2}
        \qquad\text{for } |x|\ge N,\quad 0\le t\le T.
    \]
    Therefore,
    \[
        \inf_{\T\times[0,T]}w
        =
        \inf_{\substack{|x|<N\\0\le t\le T}}w.
    \]
    There are only finitely many vertices with \(|x|<N\),
    and \(w(x,\cdot)\) is continuous for each of them.
    Hence the infimum is attained at some \((x_0,t_0)\), \(w(x_0,t_0)=-\eta.\)
    At \(t=0\), the integral inequalities give
    \[
        u(x,0)\ge f(x)\ge g(x)\ge v(x,0).
    \]
    Thus \(t_0>0\).

Evaluating
	    \(u\ge\mathcal K_\beta^f u\) and \(v\le\mathcal K_\beta^g v\) at \((x_0,t_0)\) and
	    subtracting, using \(f\ge g\) and the linearity of \(\Delta_\beta\), we have
	    \[
		    w(x_0,t_0) \ge \int_0^{t_0} e^{(s-t_0)\p^{-|x_0|}}\, \left(\Delta_\beta w(x_0,s)+ w(x_0,s) 	
		    \	p^{-|x_0|}\right)\,ds.
	    \]
	    Since \(-\eta\) is the global infimum of \(w\) on \(\T\times[0,T]\), we get
	    \[
		    -\eta=w(x_0,t_0) \ge \int_0^{t_0} e^{(s-t_0) p^{-|x_0|}}(-\eta p^{-|x_0|})\,ds
		    = -\eta\left(1-e^{-t_0 p^{-|x_0|}}\right).
	    \]
	    a contradiction. 
	    Hence \(u\ge v\) on \(\T\times[0,T]\), and since \(T>0\) was
	    arbitrary, the conclusion follows.
\end{proof}	

%%%%%%%%%%%%%%%%%%%%%%%%%%%%%%%%%%%%%%%%%%%%%%%%%%%%%%%%%%%%%%%%%%%%%%%%%%%%%%%%%%%%%%%%%%
\section{Some properties of the eigenvalues}
\label{sppe}
%%%%%%%%%%%%%%%%%%%%%%%%%%%%%%%%%%%%%%%%%%%%%%%%%%%%%%%%%%%%%%%%%%%%%%%%%%%%%%%%%%%%%%%%%%
    In this section, we study some properties of the eigenvalues and eigenfunctions.
    We recall that \(\lambda\in\mathbb{R}\) is an eigenvalue of \(\Delta_\beta\)
    if there is a bounded function \(u\colon\T\to\mathbb{R}\) (an eigenfunction) such that \(u\not\equiv 0\) and
    \begin{equation}
	    \label{prob.autovalores.def-99}
	    \begin{cases}
		    -\Delta_\beta u =\lambda u &\text{on } \T,\\
		    \lim\limits_{x\to y} u(x)=0 & \text{for } y\in \partial\T.
	    \end{cases}
    \end{equation}
    An eigenvalue \(\lambda\) is called a principal eigenvalue if it has
    a
    non-negative eigenfunction \(u\).

    \begin{lem}
	    \label{lema_fe_1}
	    Let \(\lambda\) be an eigenvalue of \(\Delta_\beta\).
	    Then, \(\lambda > 0\).
    \end{lem}

    \begin{proof}
	    We proceed by contradiction. Suppose that
	    \(\lambda \leq 0\).

	    \noindent\textit{Claim:} Any eigenfunction \(u\) corresponding to \(\lambda  \leq 0\)
	    satisfies \(u \leq 0\) in \(\T\).

	    Assuming this assertion, since \(-u\) is also an eigenfunction corresponding to
	    \(\lambda \leq 0\), we will then have \(-u \leq 0\) in \(\T\),
	    which implies \(u \equiv 0\).
	    This is a contradiction.

	    To prove the claim, suppose that
	    \[
		    \sup_{x \in \T} u(x) > 0
	    \]
	    where \(u\) is the corresponding eigenfunction.
	    Define
	    \[
		    k_0\coloneqq \min \left\{k\in\N_0\colon \exists x\in\T
		    \text{ such that }|x|=k \text{ and } u(x)>0\right\}.
	    \]
	    Choose \(x_0\in\T\) such that \(|x_0|=k_0\) and \(u(x_0)>0.\)
	    We analyze two cases:
	    \begin{itemize}
		    \item If \(x_0=\emptyset\): Using
		    \[
			    0\ge\lambda u(\emptyset)\ge -\Delta_\beta u(\emptyset)
			    =u(\emptyset)-\frac1m\sum_{i=0}^{m-1} u(\emptyset,i)
		    \]
		    and \(u(\emptyset)>0,\) we get that
		    there exists \(i\in\{0,1,\dots,m-1\}\) such that
		    \[
			    u(\emptyset,i)\ge u(\emptyset)>0.
		    \]
		    \item If \(x_0\neq\emptyset\): Using
		    \[
			    0\ge \lambda u(x_0)=-\Delta_\beta u(x_0)
			    =\left(\beta(u(x_0)-u(\hat{x}_0))
			    +(1-\beta)
			    \left(u(x_0)-\frac1m\sum_{i=0}^{m-1} u(x_0,i)\right)
			    \right)p_\beta^{-|x_0|},
		    \]
		    and that \(u(\hat{x_0})\le 0\) and \(u(x_0)>0\), there exists
		    \(i\in\{0,1,\dots,m-1\}\) such that
		    \[
			    u(x_0,i)\ge u(x_0)> 0.
		    \]
	    \end{itemize}

	    In either case, there exists a successor \(x_1\)  of \(x_0\) such that
	    \[
		    u(x_1)\ge u(x_0)>0 .
	    \]

	    Now, using
	    \[
		    0\ge \lambda u(x_1)= -\Delta_\beta u(x_1)
		    =\left(\beta(u(x_1)-u(x_0))
		    +(1-\beta)\left(u(x_1)-\frac1m\sum_{i=0}^{m-1} u(x_1,i)\right)
		    \right)p_\beta^{-|x_1|},
	    \]
	    and	\(u(x_1)\ge u(x_0)> 0,\) we have that
	    there is a successor \(x_2\) of \(x_1\) such that
	    \[
		    u(x_2)\ge u(x_1)> 0.
	    \]

	    Repeating this argument inductively, we construct a sequence
	    \(\{x_n\}_{n\ge0}\) in  \(\T\)
	    such that \(x_{n+1}\) is a
	    successor of \(x_n\) and
	    \[
		    u(x_{n+1})\ge u(x_{n})> 0
	    \]
	    for any \(n\in\N_0.\)

	    Then \(x_n\to y\in \partial\T\) and
	    \[
		    \lim_{n\to\infty} u(x_n)>0
	    \]
	    which contradicts the boundary conditions. Thus \(u(x)\le 0\) in \(\T.\)
    \end{proof}

    In the case \(\beta>0\), a consequence of the Maximum Principle
    (Theorem \ref{teo_mp}) and Lemma \ref{lema_fe_1} is that any non-negative eigenfunction \(u\)
    corresponding to a principal eigenvalue \(\lambda > 0\) is positive in \(\T\). Indeed, since
    \(-\Delta_\beta u = \lambda u \geq 0\), and given that \(\beta > 0\), we have two
    possibilities: either \(u(x) > 0\) for all \(x \in \T\),
    or \(u(x) \equiv 0\) in \(\T.\) However, since \(u\) is an eigenfunction, it cannot be identically zero.
    Thus, \(u\) must be strictly positive in \(\T\).

    \begin{lem}
	    \label{lema_fe_2}
	    If \(\lambda\) is a principal
	    eigenvalue of \(\Delta_\beta\)
	    then \(\lambda\le 1.\) Moreover, if \(\beta\in(0,\frac12)\)
	    then  \(\lambda<1.\)
    \end{lem}
    \begin{proof}
	    Let \(u\) be a non-negative eigenfunction associated to
	    \(\lambda\). We analyze two cases.

	    \noindent{\it Case \(\beta\in (0,\frac12)\).}
	    By Theorem \ref{teo_mp}, we have that \(u(x)>0\) in \(\T.\)
	    On the other hand, since
	    \[
		    \lambda u(\emptyset)=-\Delta_\beta u (\emptyset)
		    =u(\emptyset)-\frac{1}{m}\sum_{i=0}^{m-1}u(\emptyset,i),
	    \]
	    we have that
	    \[
		    (1-\lambda)u(\emptyset)=\frac{1}{m}\sum_{i=0}^{m-1}
		    u(\emptyset,i)>0.
	    \]
	    Therefore \(\lambda<1.\)

	    \noindent{\it Case \(\beta=0\) or \(\beta\in [\frac12,1)\).}
	    Let \(x\in\T\) such that \(u(x)\neq 0.\)  
	     If $x=\emptyset$ we have
	    \[
	    (1-\lambda)u(\emptyset) = \frac{1}{m} \sum_{i}u(\emptyset,i) \geq 0,
	    \]
	    so, $\lambda\leq 1$.
	    In other case, since
	    \[
		    \lambda u(x)=-\Delta_\beta u(x)=\left(
		    u(x) -\beta u(\hat{x}) -
		    \frac{1-\beta}{m} \sum_{i=0}^{m-1} u(x,i)
		    \right)
		    \p^{-|x|},
	    \]
	    we have that
	    \[
		    (1-p_\beta^{|x|}\lambda)u(x)=
		    \beta u(\hat{x}) +
		    \frac{1-\beta}{m} \sum_{i=0}^{m-1} u(x,i)\ge0.
	    \]
	    Therefore \(\lambda\le p_\beta^{-|x|}\le 1\) due to \(\beta=0\)
	    or \(\beta\in[\frac12,1)\).
    \end{proof}

    Our next lemma will be the key to show that every
    principal eigenvalue has an associated positive eigenfunction
    that is constant on each level.
    We call a function on \(\T\) \emph{radial} when it is constant on each level, and we use this terminology throughout the whole paper.

    \begin{lem}\label{autofuncion_cte_niveles}
	    Let \(\beta\in (0,1),\) \(\lambda\) be an eigenvalue of
	    \(\Delta_\beta\), and \(u\) be an eigenfunction
	    associated with \(\lambda\).
	    Then, \(\bar{u}\colon \T \to \mathbb{R}\) defined by
	    \begin{equation}
		    \label{eq:carubar}
		    \bar{u}(x) \coloneqq \frac{1}{m^{|x|}}
		    \sum_{y \in \T \colon |y| = |x|} u(y)
	    \end{equation}
	    verifies that \(\bar{u}\) is constant on each level (that is, \(\bar{u}\) is radial) and
	    \[
		    -\Delta_\beta \bar{u}(x) = \lambda \bar{u}(x)
		    \text{ in } \T
	    \]
	    with
	    $$
	    \lim\limits_{x\to y} \bar{u} (x)=0 \qquad \text{for } y\in \partial\T.
	    $$
    \end{lem}

    \begin{proof}
	    Using the linearity of the Laplacian \(\Delta_\beta\), we have
	    \[
		    \Delta_\beta \bar{u}(x) = \frac{1}{m^{|x|}}
		    \sum_{y \in \T \colon |y| = |x|}
		    \Delta_\beta u(y).
	    \]
	    It follows that
	    \[
		    \Delta_\beta \bar{u}(x) + \lambda \bar{u}(x)
		    = \frac{1}{m^{|x|}}
		    \sum_{y \in \T \colon |y| = |x|}
		    \left[ \Delta_\beta u(y) + \lambda u(y) \right] = 0.
	    \]

	    Since we have
	    $$
	    \lim\limits_{x\to y} u(x)=0 \qquad \text{for } y\in \partial\T
	    $$
	    uniformly in the branch of the tree, we also have
	    $$
	    \lim\limits_{x\to y} \bar{u} (x)=0 \qquad \text{for } y\in \partial\T.
	    $$
	    This ends the proof.
    \end{proof}

    \begin{re}
	    \label{remark_cte_niveles}
	    Suppose that \(u: \T \to \mathbb{R}\) satisfies \(\Delta_\beta u(x)
	    + \lambda u(x) \leq K\) for some \(\lambda > 0\). Then,
	    \(\bar{u}\) defined by \eqref{eq:carubar}
	    also satisfies \(\Delta_\beta \bar{u}(x) + \lambda \bar{u}(x) \leq K\).
	    Moreover, if \(K \le 0\) and \(\bar{u}\) is non-negative,
	    then \(\bar{u}\) is non-increasing
	    with respect to the levels.
	    Indeed, in the case \(\beta > 0\), using the fact that \(\bar{u}\)
	    is constant on each level, we can write for \(x = \emptyset\):
	    \[
		    \Delta_\beta \bar{u}(\emptyset) = \left(
		    \bar{u}(\emptyset, 0) -
		    \bar{u}(\emptyset) \right) \leq -\lambda \bar{u}(\emptyset) \leq 0.
	    \]
	    Therefore, \(\bar{u}(\emptyset, 0) \leq \bar{u}(\emptyset)\).
	    For any \(x \neq \emptyset\), we have
	    \[
		    \Delta_\beta \bar{u}(x) = \left[ \beta (\bar{u}(\hat{x}) - \bar{u}(x))
		    + (1 - \beta)(\bar{u}(x, 0) - \bar{u}(x)) \right] p^{-|x|} \leq
		    -\lambda \bar{u}(x) \leq
		    0.
	    \]
	    Using this inequality and
	    applying an inductive process at each level, we conclude that \(\bar{u}\) is
	    non-increasing with respect to the levels.

	    The case \(\beta = 0\) is similar.
    \end{re}

%%%%%%%%%%%%%%%%%%%%%%%%%%%%%%%%%%%%%%%%%%%%%%%%%%%%%%%%%%%%%%%%%%%%%%%%%%%%%%%%%%%%%%%%%%%%%%%
\section{\texorpdfstring{Principal eigenvalues. Cases: \(\beta=0\) or \(\beta \in [1/2,1)\)}{Principal 
eigenvalues. Cases: beta=0 or beta in (1/2,1)}}
\label{cpe1}
%%%%%%%%%%%%%%%%%%%%%%%%%%%%%%%%%%%%%%%%%%%%%%%%%%%%%%%%%%%%%%%%%%%%%%%%%%%%%%%%%%%%%%%%%%%%%%%
    We start with the characterization of the principal eigenvalue.
    In the case \(\beta=0\), the set of all principal values is \((0,1].\)

    \begin{teo}\label{casobeta=0}
	    Let \(\beta=0\). Then \(\lambda\) is a principal
	    eigenvalue of \(\Delta_\beta\)
	    if and only if \( \lambda\in (0,1].\)
    \end{teo}

    \begin{proof}
	    By Lemmas \ref{lema_fe_1} and \ref{lema_fe_2}, if
	    \(\lambda\) is a principal
	    eigenvalue of \(\Delta_\beta\),
	    then \(\lambda\in(0,1].\)

	    On the other hand,  it is straightforward to verify that
	    if \(\lambda \in (0,1],\) then
	    \[
		    u(x)=(1-\lambda)^{|x|}
	    \]
	    is an eigenfunction associated to \(\lambda\).
    \end{proof}

    In the case \(\beta\in [\frac12,1)\) the set of principal values is
    empty.

    \begin{teo}\label{casobeta>12}
	    Let \(\beta\in [\frac12,1)\). If \(\lambda\) is an
	    eigenvalue of \(\Delta_\beta\)
	    then \( \lambda\) is not principal.
    \end{teo}

    \begin{proof}
	    Suppose, for the sake of contradiction, that
	    \(\lambda\) is a principal eigenvalue of \(\Delta_\beta.\)
	    Then there is a non-negative eigenfunction \(u\)
	    associated to \(\lambda.\)  Moreover, by Lemmas \ref{lema_fe_1}
	    and \ref{lema_fe_2}, we have that \(\lambda\in(0,1].\)

	    On the other hand, using that \(u\) is
	    a non-negative eigenfunction associated to \(\lambda,\) we have
	    \begin{align*}
		    (1-\lambda)u(\emptyset)&=\frac1m\sum_{i=0}^{m-1}u(\emptyset,i)\ge0,\\
		    (1-p_\beta^{|x|}\lambda)u(x)&=\beta u(\hat{x})+
		    \frac{1-\beta}m\sum_{i=0}^{m-1}u(x,i)\ge 0 \quad \forall x\in\T
		    \setminus\{\emptyset\}.
	    \end{align*}
	    Assume that \(\beta>\frac12\), it turns out that \(p_{\beta}>1.\) Then there is
	    \(k_0\in\N\) such that if \(|x|\ge k_0,\) then
	    \(1-p_\beta^{|x|}\lambda < 0.\) Thus \(u(x)=0\) for any \(x\in \T\)
	    with \(|x|\ge k_0.\)

	    Let \(x\in \T\) be such that \(|x|= k_0.\)  Then
	    \[
		    0=(1-p_\beta^{|x|}\lambda)u(x)=\beta u(\hat{x})+
		    \frac{1-\beta}m\sum_{i=0}^{m-1}u(x,i)= \beta u(\hat{x}).
	    \]
	    Therefore \(u(\hat{x})=0.\) That is, \(u(x)=0\) for any \(x\in \T\)
	    with \(|x|\ge k_0-1.\)  By repeating this procedure,
	    it follows that \(u\equiv 0\) on \(\T,\) a contradiction.

	    For \(\beta=1/2\), we have
	    \begin{align*}
		    (1-\lambda)u(\emptyset)&=\frac1m\sum_{i=0}^{m-1}u(\emptyset,i)\ge0,\\
		    (1-\lambda)u(x)&= \frac{1}{2} u(\hat{x})+
		    \frac{1}{2}\frac{1}{m}\sum_{i=0}^{m-1}u(x,i)\ge 0 \quad \forall x\in\T
		    \setminus\{\emptyset\}.
	    \end{align*}

	    From Lemma \ref{autofuncion_cte_niveles} we can assume that \(u(x)\) only
	    depends on \(|x|\), that is,
	    $$
	    u(x) = \tilde{u} (|x|).
	    $$
	    Then, we get that
	    \begin{align*}
		    (1-\lambda) \tilde{u}(0)&= \tilde{u} (1),\\
		    (1-\lambda) \tilde{u} (|x|)&= \frac{1}{2} \tilde{u}(|x|-1)+
		    \frac{1}{2}\tilde{u}(|x|+1)  \quad \forall x\in\T
		    \setminus\{\emptyset\}.
	    \end{align*}
	    As we have that \( \tilde{u}(|x|)\to 0\) as \(|x| \to \infty,\)
	    we get that \(\tilde{u}(|x|)\) attains a positive maximum. Take \(x_0\) such that
	    $$
	    \tilde{u}(|x_0|) = \max_{\T} \tilde{u}
	    $$
	    and
	    $$
	    |x_0| = \max \left\{k \geq 0 \colon \tilde{u}(k) = \max_{j\geq 0} \tilde{u}(j) \right\}.
	    $$
	    We get that
	    $$
	    \tilde{u} (|x_0|) - \tilde{u}(|x_0|+1) = c >0.
	    $$
	    Now, from
	    $$
	    \frac{1}{2} \tilde{u}(|x_0|)+
	    \frac{1}{2}\tilde{u}(|x_0|+2) = (1-\lambda) \tilde{u} (|x_0|+1) \leq  \tilde{u} (|x_0|+1)
	    $$
	    we get
	    $$
	    \frac{1}{2} (\tilde{u}(|x_0|) - \tilde{u} (|x_0|+1) )
	    \leq \frac12 ( \tilde{u} (|x_0|+1) - \tilde{u}(|x_0|+2) ).
	    $$
	    Hence,
	    $$
	    \tilde{u} (|x_0|+1) - \tilde{u}(|x_0|+2) \geq c.
	    $$
	    Iterating this computation we obtain
	    $$
	    \tilde{u} (|x_0|+i) - \tilde{u}(|x_0|+i+1) \geq c.
	    $$
	    Then, we get
	    $$
	    \tilde{u} (|x_0|) - \tilde{u}(|x_0|+i+1) = \sum_{j=0}^i ( \tilde{u} (|x_0|+j) - \tilde{u}(|x_0|+j+1)) \geq ( i+1) c
	    $$
	    which leads to
	    $$
	    \tilde{u}(|x_0|+i+1) \leq \tilde{u} (|x_0|) -  (i+1) c < 0
	    $$
	    if \(i\) is large enough. This is a contradiction that shows that there is no principal
	    eigenvalue when \(\beta=1/2\).
    \end{proof}

%%%%%%%%%%%%%%%%%%%%%%%%%%%%%%%%%%%%%%%%%%%%%%%%%%%%%%%%%%%%%%%%%%%%%%%%%%%%%%%%%%%%%%%%%%%%%%%
\section{\texorpdfstring{Principal eigenvalues. Case: \(\beta \in (0,1/2)\)}{Principal eigenvalues. Case: beta in (0,1/2)}}
\label{cpe2}
%%%%%%%%%%%%%%%%%%%%%%%%%%%%%%%%%%%%%%%%%%%%%%%%%%%%%%%%%%%%%%%%%%%%%%%%%%%%%%%%%%%%%%%%%%%%%%%
    In this section, we focus on the case where
    \(\beta\in(0,\frac12).\)

    Motivated by \cite{BNV}, we define the set \(\A\) as
    \begin{equation}
	    \label{conjunto-A}
	    \A
	    \coloneqq
	    \left\{\lambda >0 \colon
	    \exists v \colon \T\to\R
	    \text{ with }
	    0<c < v < C,
	    \text{ satisfying }
	    \Delta_{\beta}v + \lambda v \le 0 \text{ in } \T \right\}
    \end{equation}
    and
    \begin{equation}
	    \label{lambda_1:def}
	    \lambda_1(\beta) \coloneqq \sup \A .
    \end{equation}
    Where there is no confusion, we will omit the dependence on \(\beta\) in \(\lambda_1\).

    We want to prove that  \(\lambda_1(\beta)\) is the unique principal eigenvalue of
    \(\Delta_{\beta}.\) First, we show that

    \begin{lem}
	    \label{lem:A_no_vacio}
	    Let \(\beta\in (0,\tfrac12)\).
	    Then \(\A \neq \emptyset.\)
    \end{lem}

    \begin{proof}
	    Consider the function \(v \colon \T \to \R\) defined by
	    \[
		    v(x) = 1+(|x| + a)\p^{|x|}.
	    \]
	    where \(a>0\) is a parameter that will be chosen later.
	    Note that \(v\) is uniformly positive and bounded above by some \(C>0\).
	    Since \(v\) is constant at each level of the tree, it is easy to compute
	    \(\Delta_{\beta}v(x)\) for any \(x\in \T.\)
	    In fact, we have
	    \[
		    \Delta_{\beta}v(\emptyset)
		    =
		    a(p_\beta-1) + p_\beta,
		    \quad
		    \text{ and }
		    \quad
		    \Delta_{\beta}v(x)
		    =
		    2\beta -1
		    \quad \forall x\in \T\setminus\{\emptyset\}.
	    \]
	    As \(p_\beta<1\), \(a\) can be chosen large enough such that
	    \(\Delta_{\beta}v(x)\leq 2\beta -1< 0\) for all \(x\).
	    Then, for any \(\lambda\) such that \(0<\lambda<(1-2\beta)/C\) we have
	    \[
		    \Delta_{\beta}v + \lambda v \leq 2\beta -1 + \lambda C
		    <
		    0,
	    \]
	    and it follows that \(\lambda\in \A\).
    \end{proof}

    As a consequence of the previous lemma, we have that
    \(\lambda_1(\beta)\) is well-defined and it is positive.
    Our next step is to prove that \(\lambda_1(\beta)\)
    is a lower bound for the set of eigenvalues of \(\Delta_\beta\).
    \begin{lem}
	    Let \(\beta\in (0,\tfrac12)\) and \(\lambda\) be
	    an eigenvalue of \(\Delta_\beta\) and consider
	    \(\theta \in \A\).
	    Then, \(\lambda\geq \theta\).
    \end{lem}

    \begin{proof}
	    Let \(u\) be an eigenfunction associated to \(\lambda\)
	    such that there is \(u(x_0)=1\) for a certain
	    \(x_0\in\T.\) It is bounded by definition.
	    Since \(\theta \in \A\),
	    there exists a function \(v\) such that
	    \(\Delta_\beta v +\theta v \leq 0\) and \(0<c<v<C\).
	    Moreover, we can choose it in such a way that
	    \(u\le v\) in \(\T\).

	    Note that \(w(x,t)=e^{-\theta t}v\) and
	    \(z(x,t)=e^{-\lambda t}u\) satisfy
	    \[
		    w_t(x,t)-\Delta_\beta w(x,t)\ge z_t(x,t)-
		    \Delta_\beta z(x,t)
		    \quad \text{ in } \T\times(0,\infty),\quad
		    w(x,0)\ge z(x,0) \quad \text{ in } \T,
	    \]
	    
	    We verify the boundary hypothesis of
		Theorem~\ref{teo_cp_p}.
		Fix \(T>0\). Since \(\lambda>0\) by
		Lemma~\ref{lema_fe_1} and \(\theta>0\), we have
		\[
		    w(x,t)-z(x,t)
		    \ge c e^{-\theta T}-|u(x)|
		    \qquad (0\le t\le T).
		\]
		The eigenfunction \(u\) vanishes uniformly at the boundary.
		Hence there exists \(N\) such that
		\[
		    |u(x)|\le\frac{c}{2}e^{-\theta T}
		    \qquad\text{for } |x|\ge N.
		\]
		It follows that
		\[
		    w(x,t)-z(x,t)
		    \ge\frac{c}{2}e^{-\theta T}>0
		    \qquad\text{for } |x|\ge N,\quad 0\le t\le T.
		\]
		Thus \eqref{eq:parabolic_boundary_order} holds.
	    
	    Then, by Parabolic Comparison Principle
	    (Theorem \ref{teo_cp_p}), we have that
	    \[
		    w(x,t)\ge z(x,t) \quad \text{ in }
		    \T\times(0,\infty).
	    \]
	    Therefore
	    \[
		    e^{(\lambda-\theta)t}\ge \frac{u(x)}{v(x)}
		    \quad \text{ in }
		    \T\times(0,\infty).
	    \]
	    Then
	    \[
		    e^{(\lambda-\theta)t}\ge \frac{u(x_0)}{v(x_0)}\ge
		    \frac{1}C
		    \quad \text{ for all } t\in
		    (0,\infty),
	    \]
	    which implies \(\lambda\ge \theta.\)
    \end{proof}

    An immediate consequence is the following Corollary.

    \begin{co}
	    Let \(\beta\in (0,\tfrac12)\) and \(\lambda\) be
	    an eigenvalue of \(\Delta_\beta\).
	    Then, \(\lambda\geq \lambda_1(\beta)\).
    \end{co}

    Our next result shows that any principal eigenvalue has an
    associated positive eigenfunction that is constant on each
    level and decreases with respect to the levels.

    \begin{lem}
	    \label{lema_fe_3}
	    Let \(\beta\in(0,\tfrac{1}{2}),\) and \(\lambda\in(0,1)\)
	    be a principal eigenvalue of \(\Delta_\beta.\) If \(u\)
	    is a non-negative eigenfunction associated with
	    \(\lambda,\)  then \(\bar{u}\) (defined by \eqref{eq:carubar})
	    is a positive eigenfunction associated with
	    \(\lambda\). Moreover, \(\bar{u}\)
	    is decreasing with respect to the levels.
    \end{lem}
    \begin{proof}
	    Let \(u\) be a non-negative eigenfunction associated to
	    \(\lambda\). Without loss of generality, we may assume that
	    \(u(\emptyset)=1.\)

	    By  Lemma \ref{autofuncion_cte_niveles}, the function
	    \(\bar{u}\)
	    is constant on each level and satisfies
	    \begin{equation}
		    \label{eq:auximp}
		    -\Delta_\beta \bar{u}(x) = \lambda \bar{u}(x)
		    \text{ in } \T.
	    \end{equation}
	    Furthermore, by Remark \ref{remark_cte_niveles}, we have that
	    \(\bar{u}\) is non-increasing with respect to the levels.

	    By the Maximum Principle (Theorem \ref{teo_mp}) \(u\)
	    is positive  in \(\T\). Thus, \(\bar{u}\) is also positive.

	    Next, we show that \(\bar{u}\) is strictly decreasing with respect to the levels.
	    Suppose, for contradiction, that \(\bar{u}\) is not decreasing. Define
	    \[
		    k_0= \min\left\{k\in\mathbb{N}_0\colon \exists x\in\T \text{ such that }
		    \bar{u}(x)=\bar{u}(x,0) \text{ and } |x|=k\right\}.
	    \]

	    \noindent{\it Case 1:} If \(k_0=0,\) then \(\bar{u}(\emptyset)=\bar{u}(\emptyset,0)=1.\)
	    By \eqref{eq:auximp}, we have that
	    \[
		    (1-\lambda)\bar{u}(\emptyset)= \bar{u}(\emptyset,0),
	    \]
	    which implies \(1 - \lambda = 1\). This contradicts
	    \(\lambda>0\) (Lemma \ref{lema_fe_1}).

	    \noindent{\it Case 2:} If \(k_0>0,\) there exists \(x_0\in\T\)  such that
	    \(\bar{u}(x_0)=\bar{u}(x_0,0),\) \(\bar{u}(x_0)<\bar{u}(\hat{x}_0)\)
	    and  \(|x_0|=k_0.\) By \eqref{eq:auximp},
	    we have
	    \[
		    (1-\lambda \p^{k_0})\bar{u}(x_0)=
		    \beta \bar{u}(\hat{x}_0) +(1-\beta)\bar{u}(x_0,0)>\bar{u}(x_0),
	    \]
	    which is a contradiction, because \(0<(1-\lambda \p^{k_0})<1,\) since we have
	    \(\lambda,\p\in(0,1),\) by Lemmas \ref{lema_fe_1} and \ref{lema_fe_2}.

	    Thus,  \(\bar{u}\) must be decreasing with respect to the levels.

	    To complete the proof, we show that \(\bar{u}\)
	    satisfies the boundary condition.
	    Since \(\bar{u}\) is constant at each level and decreasing,
	    the limit \(\lim_{x\to y} \bar{u}(x)\) exists for any
	    \(y\in \partial\T,\) and is independent of \(y\).

	    Denote this limit by \(c,\) and suppose \(c > 0\).
	    Then, $$\bar{u} \geq c.$$
	    Consider the function
	    $$w \coloneqq \bar{u}- \frac{c}{2}$$ which satisfies
	    \[
		    c/2 \leq w(x)\le w(\emptyset)=\bar{u}(\emptyset)- \frac{c}{2}=1- \frac{c}{2} <1.
	    \]
	    For \(\theta > 0\) sufficiently small, we have for all \(x\in \T\)
	    \[
		    \Delta_\beta w(x) + (\lambda + \theta) w(x)
		    = -\frac{c}{2}(\lambda +\theta)+\theta \bar{u}(x) \leq
		    -\frac{c}{2}\lambda+\theta\left(1-\frac{c}{2}\right)<0.
	    \]
	    This implies that \(\lambda + \theta \in \A\), which is a contradiction.
	    Therefore, \(c = 0\), and the proof is complete.
    \end{proof}

    We now show that \(\lambda_1(\beta)\)  is a non-increasing function of \(\beta.\)

    \begin{lem}
	    \(\lambda_1(\beta)\) is a non-increasing function in \((0,\frac12).\)
    \end{lem}

    \begin{proof}
	    It is enough to prove that if \(0<\gamma<\beta<\frac12\) then
	    \(\A\subseteq\mathcal{A}_{\gamma}.\)

	    If \(\lambda\in\A\) then there is \(v \colon \T\to\R\)
	    with \(0<c < v < C\), satisfying
	    \[
		    \Delta_{\beta}v + \lambda v \le 0\, \text{ in } \T.
	    \]
	    Moreover, by Remark \ref{remark_cte_niveles}, we can assume that
	    \(v\) is constant at each level of \(\T\) and it is decreasing with respect to the levels.
	    Using that \(\gamma<\beta,\) \(p_\beta\ge p_\gamma\) and \(v\) is positive
	    and non-increasing with respect to the levels,
	    for any \(x\in\T\setminus\{\emptyset\}\) we have that
	    \begin{align*}
		    0&\ge -\lambda v(x)\ge \Delta_\beta v(x)
		    =p_\beta^{-|x|}\left(\beta v(\hat{x})+(1-\beta)v(x,0)-v(x)\right)\\
		    &\ge p_\gamma^{-|x|}\left(\gamma v(\hat{x})+(1-\gamma)v(x,0)
		    -v(x)\right)+
		    p_\gamma^{-|x|}(\beta-\gamma)\left( v(\hat{x})-v(x,0)
		    \right)\\
		    &=\Delta_\gamma v(x) + p_\gamma^{-|x|}(\beta-\gamma)\left( v(\hat{x})-v(x,0)\right)\\
		    &\ge \Delta_\gamma v(x).
	    \end{align*}
	    Then \(\Delta_\gamma v(x) + \lambda v(x) \leq 0\) for all \(x\in \T\) (the condition at \(x=\emptyset\) is automatically satisfied because it does not depend on weight).
	    Hence, \(\lambda \in \mathcal{A}_\gamma.\)
    \end{proof}

    The goal now is to construct a non-negative eigenfunction associated to \(\lambda_1(\beta)\).

    \begin{lem}
	    \label{autovalores-menos-1}
	    Let  \(\beta\in (0,\tfrac12)\) and \(\lambda\in \A\) with \(\lambda <\lambda_1(\beta)\).
	    Then, there exists \(\varphi_\lambda \colon \T \to \R\) a non-increasing function with
	    respect to each level such that \(\varphi_\lambda \geq 0\), \(\varphi_{\lambda}\) is
	    constant in each level, and \(\varphi_\lambda\) solves the problem
	    \[
		    \Delta_\beta \varphi_\lambda(x) + \lambda\varphi_\lambda(x) =-1,
		    \qquad \text{ } x \in \T.
	    \]
    \end{lem}
    \begin{proof}
	    To find the required function \(\varphi_\lambda,\) we will need an auxiliary function
	    \(v_\lambda\) that controls it.
	    The first step in the proof of this lemma is the construction of \(v_\lambda\).

	    Pick \(\mu \in \A\) such that \(\lambda <\mu<\lambda_1\).
	    Then there exists a function \(v\) such that
	    \begin{equation}
		    \label{supersol-aux}
		    \Delta_\beta v + \mu v
		    \leq
		    0 \, \text{ in } \T,
	    \end{equation}
	    where \(v\) satisfies \(0<c < v <C\) for some constants \(C,c>0\).
	    From Remark \ref{remark_cte_niveles}, note that \(v\) can be taken constant in each level, and non-increasing with respect to the level.

	    From the inequality \eqref{supersol-aux} and using that \(\lambda <\mu\), we have
	    \[
		    \Delta_\beta v + \lambda v  \leq (\lambda-\mu) v <0 \, \text{ in } \T.
	    \]
	    Since \(v>c>0\) we can choose \(K>0\) large enough such that
	    \(v_\lambda\) defined as \(v_{\lambda} \coloneqq K v\) satisfies
	    \begin{equation}
		    \label{v_lambda}
		    \Delta_\beta v_{\lambda} + \lambda v_{\lambda} <-2 \, \text{ in } \T.
	    \end{equation}

	    In the second step of the proof we give a good candidate for \(\varphi_\lambda\).
	    Consider the set
	    \begin{equation}
		    \B_\lambda
		    \coloneqq
		    \left\{\varphi \colon \T\to\R \colon
		    \Delta_\beta \varphi + \lambda \varphi \geq -1 \text{ in } \T,  \, 0
		    \leq \varphi \leq v_\lambda \right\},
	    \end{equation}
	    where \(v_\lambda\) is the function constructed in the first step and that satisfies \eqref{v_lambda}.
	    Note that \(\B_{\lambda}\neq \emptyset\) since the zero function belongs to
	    \(\B_\lambda\).
	    Define \(\varphi_\lambda \colon \T\to \R\) as
	    \[
		    \varphi_\lambda (x)
		    =
		    \sup \{\varphi(x) \colon \varphi\in \B_\lambda\}
		    .
	    \]
	    It is well-defined since \(v_\lambda\) is bounded.
	    It is easy to check that \(\varphi_{\lambda}\in \B_{\lambda}\).
	    We want to show that
	    \[
		    \Delta_\beta \varphi_{\lambda} + \lambda \varphi_{\lambda} = -1 \, \text{ in } \T.
	    \]

	    We start proving that \(\varphi_\lambda < v_{\lambda}\).
	    Using that \(\Delta_\beta \varphi_{\lambda} + \lambda \varphi_{\lambda} \geq -1\)
	    and the equation \eqref{v_lambda} to estimate \(\Delta_\beta\) of the difference, we obtain
	    \[
		    \Delta_\beta(v_{\lambda}-\varphi_\lambda)\leq -1 \, \text{ in } \T.
	    \]
	    Since \(\beta\in(0,\frac12),\) the Maximum Principle (Theorem \ref{teo_mp})
	    proves that \(\varphi_\lambda < v_{\lambda}\) in \(\T\).

	    Now, assume that there is \(x\in \T\) such that
	    \(\Delta_\beta \varphi_{\lambda}(x) + \lambda \varphi_{\lambda}(x) > -1\).
	    Consider the function \(\varphi^*\colon \T\to \R\) defined by
	    \[
		    \varphi^*(y) =
		    \begin{cases}
			    \varphi_\lambda(y) & y\neq x, \\
			    \varphi_\lambda(x) + \delta & y=x,
		    \end{cases}
	    \]
	    where \(\delta>0\) will be chosen later. Note that \(\varphi^*\geq 0\).

	    First we choose \(\delta>0\) small enough such that
	    \(\varphi_\lambda(x) + \delta < v_{\lambda}(x)\) still holds.
	    
		By definition of \(\varphi^*\), we only need to check
		the inequality at \(x\), at its successors, and at
		its parent when \(x\neq\emptyset\).

	    For \(y=x\), we have that
	    \[
		    \Delta_\beta \varphi^*(x) + \lambda \varphi^*(x)
		    =\Delta_\beta \varphi_\lambda(x) +
		    \lambda \varphi_\lambda(x) + \delta(\lambda-\p^{-|x|})
		    > -1
	    \]
	    provided that \(\delta>0\) is sufficiently small.
	    
		If \(|x|\ge2\), the parent of \(x\) is not the root,
		and
		\[
		    \Delta_\beta\varphi^*(\hat{x})
		    +\lambda\varphi^*(\hat{x})
		    =
		    \Delta_\beta\varphi_\lambda(\hat{x})
		    +\lambda\varphi_\lambda(\hat{x})
		    +\frac{1-\beta}{m}\delta\p^{-|x|+1}
		    >-1.
		\]
		If \(|x|=1\), the parent is the root, so
		\[
		    \Delta_\beta\varphi^*(\emptyset)
		    +\lambda\varphi^*(\emptyset)
		    =
		    \Delta_\beta\varphi_\lambda(\emptyset)
		    +\lambda\varphi_\lambda(\emptyset)
		    +\frac{\delta}{m}
		    >-1.
		\]
		If \(x=\emptyset\), there is no parent to consider.
	    
		For each successor \(y=(x,i)\), we have
		\[
		    \Delta_\beta\varphi^*(x,i)
		    +\lambda\varphi^*(x,i)
		    =
		    \Delta_\beta\varphi_\lambda(x,i)
		    +\lambda\varphi_\lambda(x,i)
		    +\beta\delta\p^{-|x|-1}
		    >-1.
		\]
	    Then, \(\varphi^*\in \B_{\lambda}\) which contradicts the definition of \(\varphi_\lambda\)
	    as the pointwise supremum in \(\B_{\lambda}\).
	    It follows that \(\varphi_\lambda\) solves the problem
	    \(\Delta_\beta \varphi_\lambda + \lambda \varphi_\lambda=-1\).

	    To obtain a constant level function we proceed as in Lemma
	    \ref{autofuncion_cte_niveles}, and to show that it is non-increasing
	    as in Remark \ref{remark_cte_niveles}.
    \end{proof}

    For the next result we recall that
    \(\|\varphi\|_{\infty}
    =
    \sup_{x\in \T} |\varphi(x)|
    \).
    Note that the function \(\varphi_\lambda\) given in Lemma \ref{autovalores-menos-1} for some \(\lambda\) reaches \(\|\varphi_\lambda\|_{\infty}\) at the root \(\emptyset\).

    \begin{lem}
	    \label{norma-a-infinito}
	    Let \(\beta\in(0,\frac12)\) and
	    \(\{\lambda_n\}_{n\geq 2}\subset \A\) be an increasing sequence such that
	    \(\lambda_n\to\lambda_1(\beta)\). Write \(\varphi_n\) for the function given in
	    Lemma \ref{autovalores-menos-1} that solves
	    \(\Delta_\beta \varphi_n + \lambda_n \varphi_n=-1\).
	    Then
	    \[
		    \|\varphi_n\|_{\infty} \to \infty
		    \text{ when }
		    \lambda_n\to \lambda_1(\beta).
	    \]
    \end{lem}

    \begin{proof}

		Suppose, by contradiction, that
		\(\|\varphi_n\|_\infty\) does not tend to infinity.
		Then, there exist \(M>0\) and a subsequence
		\(\{\varphi_{n_j}\}_{j\ge1}\) such that $\|\varphi_{n_j}\|_\infty\le M$
		for every $j.$
		By a diagonal argument, passing to a further subsequence
		if necessary, we may assume that \(\varphi_{n_j}\)
		converges pointwise on \(\T\) to a function \(\varphi\).
		Since \(\lambda_{n_j}\to\lambda_1\) and
		\(\Delta_\beta\) involves only finitely many values
		at each vertex, we obtain
		\[
		    \|\varphi\|_\infty\le M,\qquad
		    \varphi\ge0,\qquad
		    \Delta_\beta\varphi+\lambda_1\varphi=-1.
		\]	
		    
    	    Fix \(\varepsilon, \theta>0\). Define \(\tilde{\varphi}\) as the translation of
	    \(\varphi\) by \(\varepsilon,\) \emph{i.e.}, \(\tilde{\varphi}(x)= \varphi(x)+\varepsilon\).
	    Now, we compute
	    \begin{align}
		    \Delta_\beta \tilde{\varphi}(x) + (\lambda_1+\theta)\tilde{\varphi}(x)
		    &=\Delta_\beta \varphi(x) + \lambda_1\varphi(x) + \theta \varphi(x)
		    + \lambda_1\varepsilon +\varepsilon\theta\\
		    &\leq -1 +  \theta M + \lambda_1\varepsilon +\varepsilon\theta.
	    \end{align}
	    Hence, taking \(\varepsilon, \theta>0\) small enough such that
	    \(\theta M + \lambda_1\varepsilon +\varepsilon\theta<1,\) the function \(\tilde{\varphi}\)
	    satisfies \(0<\varepsilon <\tilde{\varphi} \leq M +\varepsilon,\)
	    and
	    \[
		    \Delta_\beta \tilde{\varphi} + (\lambda_1+\theta)\tilde{\varphi} \leq 0, \, \text{ in } \T.
	    \]
	    It follows that \((\lambda_1+\theta)\in \A\), and gives a contradiction with the
	    definition of \(\lambda_1\).
    \end{proof}

    The next theorem proves that
    \(\lambda_1(\beta)\) is a principal eigenvalue.

    \begin{teo}
	    \label{existencia-autofuncion}
	    Let \(\beta\in(0,\frac12).\)
	    There exists \(u\colon \T\to \R\) such that \(u(\emptyset)=1\),
	    \(u> 0\), and
	    \[
		    \begin{cases}
			    \Delta_\beta u(x)+\lambda_1(\beta) u(x)=0 &\text{ on }\T,\\
			    \lim\limits_{x\to y} u(x) = 0 &\text{ for all } y\in \partial \T.
		    \end{cases}
	    \]
	    Furthermore, \(u\) is constant at each level and non-increasing
	    with respect to the levels.
    \end{teo}
    \begin{proof}
	    Consider the sequence \(\{u_n\}_{n\geq 2}\) defined as \(u_n(x)
	    =  \varphi_n(x)/\|\varphi_n\|_{\infty}\) where \(\{\varphi_n\}_{n\geq 2}\)
	    is the sequence of the Lemma \ref{norma-a-infinito}.
	    Each \(u_n\) is constant at each level, non-increasing and \(\|u_n\|_{\infty}=1\) for all
	    \(n\).
	    Moreover, we have \(u_n(\emptyset)=1\) for all \(n\).
	    We extract a convergent subsequence \(\{u_{n_j}\}_{j\geq 1}\) by a diagonal argument.
	    Call \(u\) for the limit function.
	    Then, \(u\colon \T\to \R\) satisfies \(u\geq 0\), \(u(\emptyset)=1\), and
	    \[
		    \Delta_\beta u(x)+\lambda_1 u(x)
		    =
		    \lim_{n\to\infty} \Delta_\beta u_n(x)+\lambda_n u_n(x)
		    =
		    \lim_{n\to\infty} \frac{-1}{\|\varphi_n\|_{\infty}}
		    =0.
	    \]
	    We have also that \(u\) is constant at each level and it is non-increasing on levels.

	    Finally, proceeding as in the proof of Lemma \ref{lema_fe_3},
	    it can be shown that \(u\) is positive and satisfies the boundary condition.
    \end{proof}

	\begin{pro}
	    \label{unicidad_autovalor_principal}
	    Let \(\beta\in(0,\tfrac12)\).
	    Then \(\lambda_1(\beta)\) is the unique principal
	    eigenvalue of \(\Delta_\beta\).
	\end{pro}
	
	\begin{proof}
	    Let \(\lambda\) and \(\nu\) be two principal eigenvalues.
	    Their associated non-negative eigenfunctions are
	    strictly positive, since \(\beta>0\).
	    By Lemma~\ref{autofuncion_cte_niveles}, we assume that these
	    eigenfunctions are radial.
	    Denote their values on level \(k\) by \(u_k\) and
	    \(v_k\), respectively.
	    Both sequences are bounded and converge to zero.
	
	    We first establish a decay estimate for a radial eigenfunction \((h_k)_{k\ge0}\) with 
	    \(\rho>0\) as an eigenvalue.
	    The eigenvalue equation gives
	    \[
	        h_0-h_1=\rho h_0, 
	        \qquad
	        \text{ and }
	        \qquad   
	        \rho p_{\beta}^k h_k
	        =
	        \beta(h_k-h_{k-1})
	        +(1-\beta)(h_k-h_{k+1}) \qquad\text{ for }k\ge1.
	    \]
	    Writing \(d_k=h_k-h_{k+1}\), we obtain
	    \[
	        d_0=\rho h_0,\qquad
	        d_k=p_{\beta}\,d_{k-1}
	        +\frac{\rho}{1-\beta}p_{\beta}^k h_k.
	    \]
	    Then,
	    \[
	        \frac{d_k}{p_{\beta}^k}
	        =
	        \rho h_0+
	        \frac{\rho}{1-\beta}\sum_{j=1}^{k}h_j
	    \]
	    Since \(h_k>0\) and \((h_k)_{k\geq 0}\) is bounded,
	    there exists \(C_h>0\) such that
	    \[
	        0<d_k\le C_h(k+1)p_{\beta}^k.
	    \]
	    Now, returning to expressing $h_k$ in terms of $d_j$, we have
	    \[
	        h_k=\sum_{j=k}^{\infty}d_j
	        \le C_h\sum_{j=k}^{\infty}(j+1)p_{\beta}^j.
	    \]
	    Hence \(h_k\leq \tilde{C}_h (k+1)p_{\beta}^k \) as \(k\to\infty\).
	    This estimate applies to both \(u_k\) and \(v_k\).
	
	    The equations for \(u\) can be written as
	    \[
	        \lambda u_0= (u_0-u_1),
	    \qquad
	    \text{ and }
	    \qquad
	        \frac{\lambda}{1-\beta} u_k
	        =
	        p_{\beta}^{-(k-1)}(u_k-u_{k-1})
	        +p_{\beta}^{-k}(u_k-u_{k+1}),
	        \qquad k\ge1.
	    \]
	    
	    The analogous identities hold for \(v\) with
	    eigenvalue \(\nu\).
	    Multiplying the equation for \(u_k\) by \(v_k\),
	    subtracting the equation for \(v_k\) multiplied
	    by \(u_k\), and summing from \(k=0\) to \(k=n\), 
	    we obtain
	    \[
	        (\lambda-\nu)u_0 v_0 + \frac{(\lambda-\nu)}{1-\beta}
	        \sum_{k=1}^{n} u_kv_k
	        =
	        p_{\beta}^{-n}(u_nv_{n+1}-u_{n+1}v_n).
	    \]
		The decay estimates imply that the right-hand side tends to zero
as \(n\to\infty\), and also show that	    
		\[
	        0<
	        \sum_{k=1}^{\infty} u_kv_k
	        <\infty.
	    \]
	    Passing to the limit yields
	    \[
			(\lambda-\nu)u_0 v_0 + \frac{(\lambda-\nu)}{1-\beta} \sum_{k=1}^{\infty} u_kv_k
			=
			0,
		\]
	    and therefore \(\lambda=\nu\).
	    Since Theorem~\ref{existencia-autofuncion}
	    proves that \(\lambda_1(\beta)\) is a
	    principal eigenvalue, it is the unique.
	\end{proof}

    Our next task is to find bounds for \(\lambda_1(\beta).\)

    \begin{pro}
	    Let \(\beta\in(0,\tfrac12).\)
	    Then
	    \[
		    \frac{(1-2\beta)^2}{\beta^2 + (1-\beta)^2}\le
		    \lambda_1(\beta)\le \frac{1-2\beta}{1-\beta}.
	    \]
    \end{pro}

    \begin{proof}
	    By Theorem \ref{existencia-autofuncion} there is a positive eigenfunction
	    \(u\) associated to \(\lambda_1\) such that \(u(\emptyset)=1,\)
	    and \(u\) is constant at each level.
	    We write \(u_k\coloneqq u(x)\) where \(|x|=k\).
	    With this notation we have
	    \begin{align}
		    \label{cota_inf_aux0}
		    \lambda_1 &= 1 - u_1\\
		    \lambda_1 u_k \p^k &=
		    \beta(u_k - u_{k-1}) + (1-\beta)(u_{k} - u_{k+1})
		    \quad
		    \forall k \ge 1.
	    \end{align}
	    Therefore, using \(u_k \leq u_0=1\) for all \(k\) and \(\p\in (0,1)\), we have
	    \begin{equation}
		    \label{cota_inf_aux1}
		    1\le \sum_{k=0}^{\infty} \p^k u_k
		    \leq
		    \sum_{k=0}^{\infty} \p^k
		    =
		    \frac{1-\beta}{1-2\beta}.
	    \end{equation}
	    Then, applying \eqref{cota_inf_aux0} and \(\lim_{k\to \infty} u_k = 0\), we get
	    \begin{equation}
		    \label{cota_inf_aux2}
		    \lambda_1 \sum_{k=0}^{\infty} \p^k u_k
		    =
		    - \beta + (1-\beta) u_1+\lambda_1
		    =
		    1 - 2\beta + \beta \lambda_1.
	    \end{equation}
	    Combining \eqref{cota_inf_aux1} and \eqref{cota_inf_aux2}, we obtain
	    \[
		    \lambda_1 \le 1 - 2\beta + \beta \lambda_1
		    \leq
		    \lambda_1\frac{1-\beta}{1-2\beta},
	    \]
	    and the bounds follow from manipulating this expression.
    \end{proof}

    \begin{re}
	    Notice that the obtained bounds,
	    \[
		    \frac{(1-2\beta)^2}{\beta^2 + (1-\beta)^2}\le
		    \lambda_1(\beta)\le \frac{1-2\beta}{1-\beta}
	    \]
	    imply that
	    $$
	    \lim_{\beta \nearrow \frac12} \lambda_1(\beta) =0, \qquad \mbox{and} \qquad
	    \lim_{\beta \searrow 0} \lambda_1(\beta) =1.
	    $$
	    
	    In fact the two bounds are asymptotically sharp: both the lower and the upper bound above
	    tend to \(1\) as \(\beta\searrow0\) and to \(0\) as \(\beta\nearrow\tfrac12\), so they pin
	    down \(\lambda_1(\beta)\) exactly at both endpoints of the range, even though a gap between
	    them remains for \(\beta\) away from the endpoints. We do not know whether \(\lambda_1(\beta)\)
	    admits a closed form for general \(\beta\in(0,\tfrac12)\).
    \end{re}

%    \begin{lem}
%	    Let \(u\colon \T\to \R\) be an eigenfunction associated to \(\lambda_1(\beta=)\) with \(u(\emptyset)=1,\) and such that \(u\) is constant at each level and decreasing with respect to the levels.
%	    Then
%	    \[
%		    \liminf_{k\to \infty} \frac{u_k-u_{k+1}}{\p^k} \geq \lambda_1(\beta).
%	    \]
%    \end{lem}
%    \begin{proof}
%	    Using that \(u\) is an eigenfunction and \(\lambda_1\) is a principal eigenvalue, we have
%	    \[
%		    0 < \lambda_1 u_k \p^k
%		    =
%		    \beta(u_k - u_{k-1})
%		    +
%		    (1-\beta)(u_{k} - u_{k+1})
%		    \qquad
%		    \text{ for } k \ge 1.
%	    \]
%	    So, since \(u_k > u_{k+1}\) we have that
%	    \[
%		    (u_{k-1}- u_k)
%		    <
%		    \frac{1-\beta}{\beta} (u_{k} - u_{k+1})
%		    =
%		    \p^{-1} (u_{k} - u_{k+1})
%		    \qquad
%		    \text{ for } k \ge 1.
%	    \]
%	    By induction we get
%	    \[
%		    u_0 -u_1 \leq \p^{-k} (u_{k} - u_{k+1}) \qquad \forall k\geq 1.
%	    \]
%	    At \(x=\emptyset\) the eigenfunction \(u\) holds
%	    \[
%		    \Delta_\beta u(\emptyset) = u_1-u_0=u_1-1 = -\lambda_1.
%	    \]
%	    Hence \(\lambda_1 \leq \p^{-k} (u_{k} - u_{k+1})\) for all \(k\geq 1\) and the result follows.
%    \end{proof}

%%%%%%%%%%%%%%%%%%%%%%%%%%%%%%%%%%%%%%%%%%%%%%%%%%%%%%%%%%%%%%%%%%%%%%%%%%%%%%%%%%%%%%%%%%%%%%%
\section{Evolution equation}\label{eq}
%%%%%%%%%%%%%%%%%%%%%%%%%%%%%%%%%%%%%%%%%%%%%%%%%%%%%%%%%%%%%%%%%%%%%%%%%%%%%%%%%%%%%%%%%%%%%%%

	We begin in Subsection~\ref{subsec:uniqueness} by proving
	uniqueness of globally bounded solutions using the Laplace
	transform in time and Lemma~\ref{lema_fe_1}.
	Subsection~\ref{exact_solution} constructs a fixed point of
	\(\mathcal K_\beta^f\) between suitable sub- and supersolutions.
	Subsection~\ref{subsec:existence_given_barrier} then shows
	that the existence of a barrier satisfying
	Definition~\ref{def:barrier} yields a solution of
	\eqref{ec.evolucion} for every bounded initial datum.
	Such a barrier is constructed in
	Subsection~\ref{subsec:general_barrier} by passing to the
	limit in finite truncations of a radial system of ODEs;
	we also establish its exponential decay in time.
	Finally, Subsection~\ref{subsec:phi_exist_unique} shows
	that, when the initial datum vanishes uniformly at the
	boundary, the solution does so uniformly for all \(t\ge0\),
	using a time-independent superharmonic majorant.

    Throughout this section, we assume \(\beta\in(0,\tfrac12)\), so that \(\p<1\).

    \subsection{Uniqueness of bounded solutions}
\label{subsec:uniqueness}

We first prove uniqueness in the class of globally bounded
solutions. The argument only requires the boundary condition
for each fixed positive time.

\begin{lem}
    \label{unicidad_general}
    Let \(f\in L^\infty(\T)\), and let
    \(u,v\in L^\infty(\T\times[0,\infty))\) be two solutions
    of \eqref{ec.evolucion} with initial datum \(f\).
    Then \(u=v\).
\end{lem}

\begin{proof}
    Set \(w=u-v\). Then \(w_t=\Delta_\beta w,\) \(w(x,0)=0,\)
    and \(w(x,t)\to 0\) along every branch as
    \(x\to\partial\T\), for each fixed \(t>0\).

    For \(\alpha>0\), define
    \[
        W_\alpha(x)
        \coloneqq
        \int_0^\infty e^{-\alpha t}w(x,t)\,dt.
    \]
    Since \(w\) is globally bounded, this integral is
    well defined and \( \|W_\alpha\|_\infty \le \|w\|_\infty/\alpha.\)

    At each fixed vertex \(x\), the operator
    \(\Delta_\beta\) involves only finitely many values.
    We can therefore interchange it with the integral.
    Moreover, \(w_t(x,\cdot)=\Delta_\beta w(x,\cdot)\)
    is bounded for each fixed \(x\). Integration by parts
    gives
    \begin{align*}
        \Delta_\beta W_\alpha(x) = \int_0^\infty e^{-\alpha t}
           \Delta_\beta w(x,t)\,dt = \int_0^\infty e^{-\alpha t}w_t(x,t)\,dt = \alpha W_\alpha(x),
    \end{align*}
    where we used \(w(x,0)=0\) and \(e^{-\alpha t}w(x,t)\to0\) as \(t\to +\infty\).

    For every \(N\), 
    \[
    \sup_{|x|\ge N}|W_\alpha(x)|
    \le
    \int_0^\infty e^{-\alpha t}
    \sup_{|x|\ge N}|w(x,t)|\,dt.
    \]
    For each fixed \(t>0\), the spatial supremum in $w(x,t)$
    tends to zero as \(N\to\infty\), by the
    boundary conditions for \(u\) and \(v\), and it is bounded by \(\|w\|_\infty\). 
    Dominated convergence therefore gives
	\[
	    \sup_{|x|\ge N}|W_\alpha(x)|\xrightarrow[N\to\infty]{}\ 0.
	\]
    Thus \(W_\alpha\) is bounded, vanishes along every
    branch, and satisfies
    \[
        -\Delta_\beta W_\alpha=-\alpha W_\alpha.
    \]
    If \(W_\alpha\not\equiv0\), this would contradict
    Lemma~\ref{lema_fe_1}. Hence \(W_\alpha\equiv0\)
    for every \(\alpha>0\).

    For each fixed \(x\), uniqueness of the Laplace
    transform implies that \(w(x,t)=0\) for almost every
    \(t\ge0\). By continuity in time, this equality holds
    for every \(t\ge0\). Therefore \(u=v\).
\end{proof}

%%%%%%%%%%%%%%%%%%%%%%%%%%%%%%%%%%%%%%%%%%%%%%%%%%%%%%%%%%%%%%%%%%%%%%%%%%%%%%%%%%%%%%%%%%%%%%%
\subsection{Existence of an exact solution}\label{exact_solution}
%%%%%%%%%%%%%%%%%%%%%%%%%%%%%%%%%%%%%%%%%%%%%%%%%%%%%%%%%%%%%%%%%%%%%%%%%%%%%%%%%%%%%%%%%%%%%%%
    Our first goal in this section is to show that for \(f\in L^\infty(\T,\mathbb{R})\) 
    there is a  solution of
    \[
	    u=\mathcal K_\beta^f u
    \]
    for any \(f\in L^\infty (T).\)

    In this subsection, functions in
    \(L^\infty(\T\times[0,\infty))\) are understood as
    specified bounded measurable representatives defined
    at every point. All inequalities and fixed-point
    identities are interpreted pointwise.

    We first show that \(\mathcal K_\beta^f\) is monotone in \(u.\)

    \begin{lem}
	    \label{Kbeta_monotona}
	    If \(u,u'\in L^\infty(\T\times[0,\infty))\) and \(u\le u'\) on \(\T\times[0,\infty),\) then
	    \[
		    \mathcal K_\beta^f u \le \mathcal K_\beta^f u' \text{ on }\T\times[0,\infty).
	    \]
    \end{lem}

    \begin{proof}
	    For \(x=\emptyset\),
	    \[
		    \Delta_\beta u(x,s)+ u(x,s) = \frac{1}{m}\sum_{i=0}^{m-1}u((x,i),s)\le  \frac{1}{m}\sum_{i=0}^{m-1}u'((x,i),s)
		    =\Delta_\beta u'(x,s)+ u'(x,s),
	    \]
	    and for \(x\neq\emptyset\), since \(\beta,1-\beta\ge0\),
	    \begin{align*}
		    \Delta_\beta u(x,s)+\p^{-|x|} u(x,s) &= \p^{-|x|}
		    \left(\beta u(\hat x,s)+\frac{1-\beta}{m}\sum_{i=0}^{m-1}u((x,i),s)\right)\\
		    &\le \p^{-|x|}
		    \left(\beta u'(\hat x,s)+\frac{1-\beta}{m}\sum_{i=0}^{m-1}u'((x,i),s)\right)\\
		    &=\Delta_\beta u'(x,s)+\p^{-|x|} u'(x,s).
	    \end{align*}
	    Integrating against \(e^{(s-t) p^{-|x|}}\ge0\) in \eqref{Kbeta},
	    \begin{align*}
		    \mathcal K_\beta^f u(x,t)
		    &= e^{-t p^{-|x|}}f(x)+\int_0^t e^{(s-t) p^{-|x|}}\left(\Delta_\beta u(x,s)+ p^{-|x|} u(x,s)\right)\,ds\\
		    &\le e^{-t p^{-|x|}}f(x)+\int_0^t e^{(s-t) p^{-|x|}}\left(\Delta_\beta u'(x,s)+ p^{-|x|} u'(x,s)\right)\,ds\\
		    &= \mathcal K_\beta^f u'(x,t).
	    \end{align*}
    \end{proof}

    \begin{co}\label{minsupersol}
	    If \(w_1,w_2\) are such that
	    \(w_i\ge\mathcal K_\beta^ f w_i\) \(i=1,2\),
	    then \(w=\min\{w_1,w_2\}\)
	    also satisfies that \(w\ge\mathcal K_\beta^f w.\)
    \end{co}

    \begin{proof}
	    Since \(w\le w_i\) \(i=1,2\), by Lemma \ref{Kbeta_monotona}, we have
	    \(\mathcal{K}_\beta^fw\le 
	    \mathcal{K}_\beta^fw_i\le w_i\) \(i=1,2\).
	    Then
	    \(\mathcal K_\beta^f w\le w\).
    \end{proof}

    Our next step is to show that \(\mathcal{K}_\beta^f\) admits a fixed point between appropriate barriers.

    \begin{teo}
	    \label{existencia_sin_cp}
	    Let \(f\in L^\infty(\T)\), and let \(w_1,w_2\in L^\infty(\T\times[0,\infty))\) satisfy
	    \[
		    w_2\le w_1, \qquad w_1\ge\mathcal K_\beta^f w_1,\qquad w_2\le\mathcal K_\beta^f w_2
		    \qquad\text{on }\T\times[0,\infty).
	    \]
	    Then, the function
	    \[
		    u(x,t)\coloneqq\sup\{w(x,t)\colon w\in\S\}, \text{ where }
		    \S\coloneqq \left\{w\in L^\infty(\T\times[0,\infty))\colon w\le w_1,\ w\le\mathcal K_\beta^f w\right\},
	    \]
	    
	    is bounded and measurable, satisfies \(w_2\le u\le w_1\),
	    and solves
	    \[
	    u(x,t)=\mathcal K_\beta^f u(x,t)
	    \qquad\text{on }\T\times[0,\infty).
	    \]
	    
    \end{teo}

\begin{proof}
    Define the sequence $u_0=w_1,$ $u_{n+1}=\mathcal K_\beta^f u_n$.
    We first claim that for every $n\ge 0$
    \[
        w_2\le u_{n+1}\le u_n\le w_1.
    \]
    Now we proceed by induction. 
    The hypotheses and
    Lemma~\ref{Kbeta_monotona} give
    \[
        w_2\le\mathcal K_\beta^f w_2
        \le\mathcal K_\beta^f w_1
        =u_1\le w_1=u_0.
    \]
    If \(w_2\le u_n\le u_{n-1}\), monotonicity yields
    \[
        w_2\le\mathcal K_\beta^f w_2
        \le\mathcal K_\beta^f u_n
        =u_{n+1}
        \le\mathcal K_\beta^f u_{n-1}
        =u_n,
    \]
    and the claim follows..

    Each \(u_n\) is measurable. Moreover, the sequence
    is pointwise non-increasing and bounded below by
    \(w_2\). Hence the pointwise limit
    \[
        \tilde{u}(x,t)\coloneqq\lim_{n\to\infty}u_n(x,t)
    \]
    exists, is measurable, and satisfies $w_2\le \tilde{u }\le w_1.$
    In particular, \(\tilde{u}\) is globally bounded.

    We next prove that \(\tilde{u}=\mathcal K_\beta^f \tilde{u}\).
    Choose \(M<\infty\) such that
    \[
        |w_1(x,t)|,\ |w_2(x,t)|\le M
        \qquad\text{for every }(x,t).
    \]
    Then \(|u_n(x,t)|\le M\) for all \(n,x,t\).
    Fix \(x\in\T\) and \(t\ge0\).
    Since \(\Delta_\beta\) involves only finitely many
    values at each vertex,
    \[
        \Delta_\beta u_n(x,s)+p_\beta^{-|x|} u_n(x,s)
        \longrightarrow
        \Delta_\beta \tilde{u} (x,s)+ p_\beta^{-|x|} \tilde{u}(x,s)
    \]
    for every \(s\ge0\).
    Also,
    \[
        \left|
        \Delta_\beta u_n(x,s)+ p_\beta^{-|x|}  u_n(x,s)
        \right|\le p_{\beta}^{-|x|}M.
    \]
    Thus dominated convergence in \eqref{Kbeta} gives
    \[
        \tilde{u}(x,t)
        =\lim_{n\to\infty}u_{n+1}(x,t)
        =\lim_{n\to\infty}\mathcal K_\beta^f u_n(x,t)
        =\mathcal K_\beta^f \tilde{u}(x,t).
    \]

    Finally, we identify \(\tilde{u}\) with the supremum in
    the statement. Let \(w\in\S\).
    Since \(w\le w_1=u_0\), it follows by induction that \(w\le u_n,\) for every \(n\).
    Passing to the limit, we obtain \(w\le \tilde{u}\).
    Therefore,
    \[
        \sup_{w\in\S}w(x,t)\le \tilde{u}(x,t).
    \]
    Conversely, \(\tilde{u}\le w_1\) and
    \(\tilde{u}=\mathcal K_\beta^f \tilde{u}\), so \(\tilde{u}\in\S\).
    Consequently,
    \[
        \tilde{u}(x,t)\le\sup_{w\in\S}w(x,t).
    \]
    Hence \(u=\tilde{u}\).
\end{proof}

%%%%%%%%%%%%%%%%%%%%%%%%%%%%%%%%%%%%%%%%%%%%%%%%%%%%%%%%%%%%%%%%%%%%%%%%%%%%%%%%%%%%%%%%%%%%%%%
\subsection{Existence, given a barrier}
\label{subsec:existence_given_barrier}
%%%%%%%%%%%%%%%%%%%%%%%%%%%%%%%%%%%%%%%%%%%%%%%%%%%%%%%%%%%%%%%%%%%%%%%%%%%%%%%%%%%%%%%%%%%%%%%
    Theorem \ref{existencia_sin_cp} already produces,
    between any two ordered barriers \(w_2\le w_1\), an exact
    solution of \(u=\mathcal K_\beta^fu\).
    What is still missing, for that solution to actually solve
    \eqref{ec.evolucion} is control of its behavior at
    \(\partial\T\). We isolate exactly what is needed,
    prove existence for every datum granting
    that such a control exists, and construct it afterwards.

    \begin{de}[Barrier]
	    \label{def:barrier}
	    We say that \(\phi\in L^\infty(\T\times[0,\infty))\)
	    is a \emph{barrier} if
	    \begin{enumerate}
		    \item \(\phi=\mathcal K_\beta^1\phi\)
		    exactly on \(\T\times[0,\infty)\);
		    \item \(0\le\phi\le1\) on \(\T\times[0,\infty)\);
		    \item  for every \(t>0\), 
		    \(\sup\limits_{|x|\ge N}|\phi(x,t)|\to 0\) as \(N\to\infty\).
	    \end{enumerate}
    \end{de}

    The whole difficulty of the general problem reduces, then,
    to producing a single radial function with these three properties;
    this is exactly what we achieve in
    Subsection~\ref{subsec:general_barrier} below. Granting for now that a barrier exists,
    we obtain existence for every datum
    \(f\in L^\infty(\T)\).

    \begin{lem} \label{sandwich_f}
	    Let \(\phi\) be a barrier and let \(f\in L^\infty(\T)\).
	    With \(w_1\coloneqq\|f\|_\infty\,\phi\),
	    \(w_2\coloneqq-\|f\|_\infty\,\phi\), the solution
	    \(u_f\) given by Theorem \ref{existencia_sin_cp}
	    satisfies \(u_f=\mathcal K_\beta^fu_f\) exactly,
	    and
	    \[
		    -\|f\|_\infty\,\phi(x,t)\ \le\ u_f(x,t)\ \le\
		    \|f\|_\infty\,
		    \phi(x,t) \qquad\text{on }\T\times[0,\infty).
	    \]
    \end{lem}

    \begin{proof}
	    Since \(\phi=\mathcal K_\beta^1\phi\) exactly and
	    \(-\|f\|_\infty\le f\le\|f\|_\infty\), monotonicity of
	    \(\mathcal K_\beta^{(\cdot)}\) in the datum gives
	    \(w_2\le\mathcal K_\beta^fw_2\) and \(w_1\ge\mathcal
	    K_\beta^fw_1\). Theorem \ref{existencia_sin_cp},
	    applied to this pair,
	    gives an exact solution
	    \(u_f=\mathcal K_\beta^fu_f\) with \(w_2\le u_f\le w_1\).
    \end{proof}

    Note that this single lemma already does two jobs at once: it produces
    the candidate solution \(u_f\), and
    it bounds it by the barrier.
    Everything else in this subsection is a direct consequence.

	\begin{teo}
	    \label{teo_exist_unic_v2}
	    Let \(\beta\in(0,\tfrac12)\) and \(f\in L^\infty(\T)\).
	    Then there exists a unique solution
	    \(u_f\in L^\infty(\T\times[0,\infty))\)
	    of \eqref{ec.evolucion}.
	\end{teo}    
    
    \begin{proof}
	    Let \(\phi\) be a barrier and let \(u_f\) be the solution given by
	    Lemma~\ref{sandwich_f}, so \(u_f=\mathcal K_\beta^fu_f\) exactly and
	    \(|u_f|\le\|f\|_\infty\,\phi\). By Definition~\ref{def:barrier} (3),
	    for every \(y\in\partial\T\) and every fixed \(t\in(0,\infty)\),
	    \(\phi(x,t)\to0\) as \(x\to y\); hence also \(u_f(x,t)\to0\) as
	    \(x\to y\) for every fixed \(t\in(0,\infty)\), exactly the boundary
	    condition \eqref{ec.evolucion} requires. Together with
	    \(u_f=\mathcal K_\beta^fu_f\), this shows that \(u_f\) solves \eqref{ec.evolucion}.
	    
	    Uniqueness follows from Lemma~\ref{unicidad_general}.    
    \end{proof}

	\begin{re}
	    \label{re_uniqueness_boundary}
	    For a general bounded initial datum, the boundary
	    condition cannot hold uniformly down to \(t=0\),
	    since \(u_f(x,0)=f(x)\) need not vanish at the boundary.
	    This does not prevent uniqueness:
	    Lemma~\ref{unicidad_general} only uses the boundary
	    condition at each fixed positive time.
	\end{re}

%%%%%%%%%%%%%%%%%%%%%%%%%%%%%%%%%%%%%%%%%%%%%%%%%%%%%%%%%%%%%%%%%%%%%%%%%%%%%
\subsection{The general barrier}
\label{subsec:general_barrier}
%%%%%%%%%%%%%%%%%%%%%%%%%%%%%%%%%%%%%%%%%%%%%%%%%%%%%%%%%%%%%%%%%%%%%%%%%%%%%

    We now carry out the construction promised in
    Subsection~\ref{subsec:existence_given_barrier}:
    a single
    radial function \(\phi\) satisfying Definition~\ref{def:barrier} (as defined in
    Section~\ref{sppe}, see Lemma~\ref{autofuncion_cte_niveles}). The
    natural strategy is to look for it
    among radial functions from the start, so we first identify the equation
    any such function must satisfy.

    \subsubsection{The radial ODE}

    \begin{re}\label{Kbeta_radial}
	    The operator \(\mathcal K_\beta^1\) does not see individual vertices,
	    only levels: if \(u\) is constant on
	    each level of \(\T\), so is \(\mathcal K_\beta^1u\),
	    since \(\Delta_\beta u(x)\) depends on \(x\) only through
	    \(|x|=k\) and the values of \(u\) on levels \(k-1,k,k+1\),
	    which are the same for every \(x\) with \(|x|=k\).
	    This suggests that, even when \(u\)
	    is \emph{not} constant on levels, its level-averages
	    should still obey a
	    closed, purely radial, equation.
    \end{re}

    \begin{lem}\label{radial_reduction}
	    Let \(u\colon\T\times[0,\infty)\to\R\)
	    be bounded and satisfy \(u=\mathcal K_\beta^1u\)
	    on \(\T\times[0,\infty)\).
	    For \(k\ge0\), set
	    \[
		    \bar u_k(t)\coloneqq\frac1{m^k}\sum_{|x|=k}u(x,t).
	    \]
	    Then \(\bar u=(\bar u_k)_{k\ge0}\) satisfies
	    \begin{equation}\label{ode}
		    \dot{\bar u}_0(t)=\bar u_1(t)-\bar u_0(t), \qquad
		    \dot{\bar u}_k(t)=\p^{-k}\left(\beta\,
		    \bar u_{k-1}(t)+(1-\beta)\,\bar u_{k+1}(t)
		    -\bar u_k(t)\right)\ \ (k\ge1),
	    \end{equation}
	    with \(\bar u_k(0)=1\) for every \(k\ge0\).
    \end{lem}

    \begin{proof}
	    Since \(u=\mathcal K_\beta^1u\), differentiating in \(t\)
	    gives the pointwise identity \(\dot u(x,t)=\Delta_\beta u(x,t)\)
	    for every \(x\in\T\), with \(u(x,0)=1\).

	    At \(x=\emptyset\) (\(k=0\)),
	    \(\dot u(\emptyset,t)
	    =\frac1m\sum_{|y|=1}u(y,t)-u(\emptyset,t)
	    =\bar u_1(t)-\bar u_0(t)\),
	    since \(\bar u_0(t)=u(\emptyset,t)\).

	    For \(k\ge1\) and \(|x|=k\),
	    \[
		    \dot u(x,t)=\p^{-k}\left(\beta\,u(\hat x,t)
		    +\frac{1-\beta}m\sum_i u\left((x,i),t\right)-u(x,t)\right).
	    \]
	    Summing over the \(m^k\) vertices \(x\) of level \(k\)
	    and dividing by \(m^k\):
	    \begin{itemize}
		    \item Each vertex \(y\) of level \(k-1\)
		    is the parent \(\hat x\) of exactly \(m\) vertices \(x\)
		    of level \(k\), so
		    \[
			    \dfrac1{m^k}\displaystyle\sum_{|x|=k}u(\hat x,t)
			    =\dfrac1{m^{k-1}}\displaystyle\sum_{|y|=k-1}u(y,t)
			    =\bar u_{k-1}(t).
		    \]
		    \item As \(x\) ranges over level \(k\) and \(i\) over
		    \(1,\dots,m\), the pair \((x,i)\) ranges exactly once over every
		    vertex of level \(k+1\)
		    (each such vertex has a unique parent and a unique child-index), so
		    \[
			    \dfrac1{m^{k+1}}\displaystyle
			    \sum_{|x|=k}\sum_i u\left((x,i),t\right)=\bar u_{k+1}(t).
		    \]
	    \end{itemize}
	    Combining these two identities with the coefficients \(\beta\)
	    and \(1-\beta\) gives \eqref{ode} for \(k\ge1\).
	    Finally,
	    \[
		    \bar u_k(0)=\frac1{m^k}\sum_{|x|=k}u(x,0)
		    =\frac1{m^k}\sum_{|x|=k}1=1.
	    \]
    \end{proof}

    Observe that \eqref{ode} is exactly the equation governing the average
    of any solution with datum \(f\equiv1\).
    We now construct a solution of \eqref{ode} directly,
    and take it as the barrier \(\phi\).

%%%%%%%%%%%%%%%%%%%%%%%%%%%%%%%%%%%%%%%%%%%%%%%%%%%%%%%%%%%%%%%%%%%%%%%%%%%%%%%%%%%%%%%%%%%%%%%
\subsubsection{A finite auxiliary system.}
%%%%%%%%%%%%%%%%%%%%%%%%%%%%%%%%%%%%%%%%%%%%%%%%%%%%%%%%%%%%%%%%%%%%%%%%%%%%%%%%%%%%%%%%%%%%%%%
    For \(n\ge1\), let \(q^{(n)}=(q^{(n)}_0,\dots,q^{(n)}_n)\)
    be the unique solution of the finite linear ODE system
    \begin{equation}\label{odeq}
	    \dot q^{(n)}(t) = A_n\, q^{(n)}(t),
	    \qquad q^{(n)}(0)=\mathbf 1_{n+1},
    \end{equation}
    where \(A_n=\left((A_n)_{i,j}\right)_{0\le i,j\le n}\) is given by
    \[
	    (A_n)_{i,j}=
	    \begin{cases}
		    -1 & \text{if } i=j=0,\\
		    1 & \text{if } i=0,\ j=1,\\
		    -\p^{-i} & \text{if } 1\le i\le n,\ j=i,\\
		    \p^{-i}\beta & \text{if } 1\le i\le n,\ j=i-1,\\
		    \p^{-i}(1-\beta) & \text{if } 1\le i\le n-1,\ j=i+1,\\
		    0 & \text{otherwise.}
	    \end{cases}
    \]

    Our next result shows that \(-A_n\) is a non-singular
    \(M\)-matrix, that is, \(-A_n=sI-B\) for some \(s>0\)
    and \(B\ge0\) with \(s>\rho(B)\), the spectral radius of \(B\).
    Consequently \((-A_n)^{-1}\ge0\) entrywise,
    and every eigenvalue of \(A_n\) has negative real part.

    \begin{lem}
	    \label{Mmatrix_An}
	    For every \(n\ge1\), \(-A_n\) is a non-singular M-matrix.
    \end{lem}

    \begin{proof}
	    \(-A_n\) has non-negative diagonal entries and non-positive off-
	    diagonal entries (a \(Z\)-matrix), and its
	    row sums are
	    \[
		    \sum_j (-A_n)_{i,j}=0\ \ (0\le i\le n-1),
		    \qquad \text{ and } \qquad
		    \sum_j(-A_n)_{n,j}=\p^{-n}(1-\beta)>0,
	    \]
	    so \(-A_n\) is weakly diagonally dominant,
	    with strict dominance in row \(n\). 
	    
	    Moreover, the entries
		\[
		    (A_n)_{0,1}=1>0,
		    \qquad
		    (A_n)_{i,i+1}=\p^{-i}(1-\beta)>0
		    \quad(1\le i\le n-1)
		\]
		show that every index \(i<n\) has a directed path
		to \(n\) in the directed graph of \(A_n\).
	    
	    By \cite[Theorem 2.3, condition (\(L_{32}\))]{berman},
	    \(-A_n\) is a non-singular \(M\)-matrix.
    \end{proof}

    Now we show some properties of \(q^{(n)}.\)
    \begin{lem}
	    \label{qn_props}
	    For every \(n\ge1\):
	    \begin{enumerate}
		    \item \(e^{tA_n}\ge0\) entrywise for every \(t\ge0\);
		    \item \(0\le q^{(n)}(t)\le\mathbf1\) for every \(t\ge0\);
		    \item \(t\mapsto q^{(n)}(t)\) is non-increasing, componentwise;
		    \item \(q^{(n+1)}_k(t)\ge q^{(n)}_k(t)\) \
		    for every \(0\le k\le n\) and \(t\ge0\).
	    \end{enumerate}
    \end{lem}

    \begin{proof}
	    (1) Let \(c\coloneqq\max_i(-A_n)_{i,i}\).
	    Then \(A_n+cI\ge0\) entrywise, so
	    \[
		    e^{t(A_n+cI)}=\sum_{j=0}^\infty\dfrac{t^j(A_n+cI)^j}{j!}\ge0,
	    \]
	    and therefore \(e^{tA_n}=e^{-ct}e^{t(A_n+cI)}\ge0\).

	    \medskip

	    \noindent(2) By (1),
	    \(e^{tA_n}\ge0\), so \(q^{(n)}(t)\ge0\).
	    On the other hand \(A_n\mathbf1\le0\), so by
	    the Duhamel principle,
	    \[
		    \mathbf1-q^{(n)}(t)
		    = \int_0^t e^{(t-s)A_n}(-A_n\mathbf1)\,ds \ge0,
	    \]
	    since the integrand is a product of a
	    non-negative matrix and a non-negative vector.

	    \medskip

	    \noindent(3) Since \(e^{tA_n}\ge0\) is order preserving and,
	    by (2), \(q^{(n)}(h)\le\mathbf1=q^{(n)}(0)\)
	    for every \(h\ge0\), applying \(e^{tA_n}\)
	    gives \(q^{(n)}(t+h)\le q^{(n)}(t)\) for all \(t,h\ge0\).

	    \medskip

	    \noindent (4) Let \(w_k(t)\coloneqq q^{(n+1)}_k(t)\)
	    for \(0\le k\le n\). For \(k\le n-1\), \(w\) satisfies
	    exactly the same equation as \(q^{(n)}\); at \(k=n\),
	    \[
		    \dot w_n(t) = \p^{-n}\left(
		    \beta w_{n-1}(t)+(1-\beta)q^{(n+1)}_{n+1}(t)-w_n(t)\right).
	    \]
	    That is,
	    \[
		    \dot w(t) =A_n w + g \quad \text{ and } \quad w(0)=\mathbf1,
	    \]
	    where \(g_k=0\) if \(k\le n-1\) and
	    \(g_n=\p^{-n}(1-\beta)q^{(n+1)}_{n+1}(t)\ge0\). By the Duhamel
	    principle together with \(e^{tA_n}\ge0\),
	    \(w(t)\ge q^{(n)}(t)\), that is
	    \[
		    q^{(n+1)}_k(t)\ge q^{(n)}_k(t)
	    \]
	    for \(0\le k\le n\).
    \end{proof}

    By Lemma~\ref{qn_props}, for each fixed \(k\),
    \(n\mapsto q_k^{(n)}(t)\) is non-decreasing (for \(n\ge k\))
    and bounded above by \(1\). We define
    \[
	    q_k(t)\coloneqq\lim_{n\to\infty}q_k^{(n)}(t)\in[0,1],
	    \quad q_k(0)=1, \quad\text{ and }\quad
	    q(t)\coloneqq(q_k(t))_{k\ge0}.
    \]

    \begin{lem}
	    \label{q_solves}
	    \(q=(q_k)_{k\ge0}\) satisfies
	    \eqref{ode}, with \(q_k(0)=1\) for every \(k\ge0\).
    \end{lem}

\begin{proof}
    Fix \(k\ge1\) and take \(n>k+1\).
    Then \(q_k^{(n)}\) satisfies exactly the untruncated
    equation \eqref{ode}, involving only
    \(q_{k-1}^{(n)},q_k^{(n)},q_{k+1}^{(n)}\).
    Therefore,
    \[
        q_k^{(n)}(t)
        = e^{-\p^{-k}t}
        +\int_0^t \p^{-k}e^{-\p^{-k}(t-s)}
        \left(
            \beta q_{k-1}^{(n)}(s)
            +(1-\beta)q_{k+1}^{(n)}(s)
        \right)\,ds.
    \]
    By Lemma~\ref{qn_props}, for each fixed \(j\),
    \(n\mapsto q_j^{(n)}(s)\) increases to \(q_j(s)\)
    as \(n\to\infty\), with \(0\le q_j^{(n)}(s)\le1\)
    uniformly in \(n\) and \(s\).
    By the Dominated Convergence Theorem, we may pass
    to the limit inside the integral, obtaining
    \[
        q_k(t)
        = e^{-\p^{-k}t}
        +\int_0^t \p^{-k}e^{-\p^{-k}(t-s)}
        \left(
            \beta q_{k-1}(s)
            +(1-\beta)q_{k+1}(s)
        \right)\,ds.
    \]

    At the root, the corresponding identity is
    \[
        q_0^{(n)}(t)
        =e^{-t}
        +\int_0^t e^{-(t-s)}q_1^{(n)}(s)\,ds.
    \]
    Passing to the limit in the same way gives
    \[
        q_0(t)
        =e^{-t}
        +\int_0^t e^{-(t-s)}q_1(s)\,ds.
    \]
    Thus the radial function \(q(x,t)=q_{|x|}(t)\)
    satisfies \(q=\mathcal K_\beta^1q\).
    By the equivalence between the integral and
    differential formulations, \(q\) solves
    \eqref{ode}, with \(q_k(0)=1\) for every \(k\ge0\).
\end{proof}

    We have that \(q\) is non-increasing by levels.

    \begin{lem} \label{q_monotona_k}
	    \(q_k(t)\ge q_{k+1}(t)\) for every \(k\ge0\), \(t\ge0\).
    \end{lem}
    \begin{proof}
	     Fix \(n\ge2\) and, for \(0\le k\le n-1\),
	    set \(d_k(t)\coloneqq q_k^{(n)}(t)-q_{k+1}^{(n)}(t)\), so
	    \(d_k(0)=0\).
	    Subtracting the equations for consecutive components, 
	    and treating the root and the last component separately,
	    we find that \(d=(d_0,\dots,d_{n-1})\) solves the linear system
	    \[
		    \dot d(t) = B_n\,d(t) + g(t), \qquad d(0)=0,
	    \]
	    where \(B_n=\left((B_n)_{i,j}\right)_{0\le i,j\le n-1}\)
	    is given by
	    \[
		    (B_n)_{i,j}=
		    \begin{cases}
			    -\left(1+\p^{-1}\beta\right) & \text{if } i=j=0,\\
			    \p^{-1}(1-\beta) & \text{if } i=0,\ j=1,\\
			    -\beta\p^{-(i+1)}-(1-\beta)\p^{-i} & \text{if }
			    1\le i\le n-2,\ j=i,\\
			    -\left(\p^{-(n-1)}(1-\beta)+\p^{-n}\right)
			    & \text{if } i=j=n-1,\\
			    \p^{-i}\beta & \text{if } 1\le i\le n-1,\ j=i-1,\\
			    \p^{-(i+1)}(1-\beta) & \text{if } 1\le i\le n-2,\ j=i+1,\\
			    0 & \text{otherwise,}
		    \end{cases}
	    \]
	    and
	    \(g(t)\coloneqq\left(0,\dots,0,\p^{-n}(1-\beta)
	    \,q^{(n)}_{n-1}(t)\right)\).

	    All off-diagonal entries of \(B_n\) are non-negative,
	    so \(B_n\) is a Metzler matrix and \(e^{tB_n}\ge0\)
	    entrywise, as in Lemma~\ref{qn_props}(1).
	    By Lemma~\ref{qn_props}(2), \(q^{(n)}_{n-1}(t)\ge0\), so
	    \(g(t)\ge0\) entrywise. By the Duhamel principle,
	    \[
		    d(t) = \int_0^t e^{(t-s)B_n}\,g(s)\,ds \ \ge 0.
	    \]
	    Therefore \(q_k^{(n)}(t)\ge q_{k+1}^{(n)}(t)\)
	    for every \(0\le k\le n-1\), \(t\ge0\). 
	    Fixing \(k\geq 0\) and letting \(n\to\infty\), we obtain \(q_k(t)\ge q_{k+1}(t)\).
    \end{proof}

    We now show that \(q\) is in fact the smallest solution of \(w=\mathcal K_\beta^1w\).

    \begin{teo}\label{q_minimal}
	    Let \(w\in L^{\infty}(\T\times[0,\infty))\) be non-negative
	    and satisfy \(w=\mathcal K_\beta^1w\)
	    exactly on \(\T\times[0,\infty)\). Then
	    \[
		    q_{|x|}(t)\ \le\ w(x,t)
		    \qquad \text{for every }(x,t)\in\T\times[0,\infty).
	    \]
    \end{teo}

    \begin{proof}
	    Fix \(n\ge1\) and set \(\T\!\!\!\!^n\coloneqq\{x\in\T:|x|\le n\}.\) Consider, the linear system
	    \[
		    \dot u(x,t)=\Delta_\beta u(x,t)\ \
		    (x\in\T\!\!\!\!^n,\ |x|<n), \qquad
		    \dot u(x,t)=\p^{-n}\left(\beta u(\hat x,t)-u(x,t)\right)\ \ (|x|=n),
	    \]
	    that is the full equation on the interior of \(\T\!\!\!\!^n\), with the outgoing term to the children
	    dropped at the boundary level \(n\), the same truncation used to define \(q^{(n)}\) in
	    Lemma~\ref{Mmatrix_An}, now applied vertex by vertex instead of level by level. Write
	    \(\dot u=C_nu\), exactly as in the proof of Lemma~\ref{Mmatrix_An}, \(-C_n\) is a non-singular
	    \(M\)-matrix, so \(e^{tC_n}\ge0\) entrywise for every \(t\ge0\).

	    Let \(Q^{(n)}(x,t)\) be the solution of this system with \(Q^{(n)}(x,0)=1\) for every
	    \(x\in\T\!\!\!\!^n\). Since both the system and the datum are invariant under the automorphisms of
	    \(\T\!\!\!\!^n\) fixing the root levelwise, \(Q^{(n)}(x,t)\) depends on \(x\) only through \(|x|\); by
	    uniqueness for this finite linear ODE, \(Q^{(n)}(x,t)=q^{(n)}_{|x|}(t)\), the same truncated
	    radial solution as in Lemma~\ref{q_solves}.

	    Set \(z(x,t)\coloneqq Q^{(n)}(x,t)-w(x,t)\) for \(x\in\T\!\!\!\!^n\). Since \(w=\mathcal K_\beta^1w\)
	    exactly, \(w(x,0)=1\) for every \(x\), so \(z(\cdot,0)=0\). For \(|x|<n\), \(w\) satisfies the
	    untruncated equation \(\dot w(x,t)=\Delta_\beta w(x,t)\), which coincides with \(Q^{(n)}\)'s
	    equation there; at \(|x|=n\), \(w\)'s equation has the extra non-negative term
	    \((1-\beta)\p^{-n}\overline{w^+}(x,t)\) (the average over the children of \(x\)) that
	    \(Q^{(n)}\)'s equation lacks, so subtracting, \(\dot z(t)\le C_nz(t)\) entrywise on \(\T\!\!\!\!^n\).
	    {Writing \(r(t)\coloneqq C_nz(t)-\dot z(t)\ge0\) entrywise, Duhamel's formula gives
	    \[
		    z(t)=e^{tC_n}z(0)-\int_0^te^{(t-s)C_n}r(s)\,ds,
	    \]
	    since \(e^{tC_n}\ge0\) entrywise, as noted above, and \(z(\cdot,0)=0\), both terms on the
	    right-hand side are non-positive entrywise, so \(z(t)\le0\)}, that is, \(Q^{(n)}(x,t)\le w(x,t)\), i.e.\
	    \(q^{(n)}_{|x|}(t)\le w(x,t)\), for every \(x\in\T\!\!\!\!^n\), \(t\ge0\). Letting \(n\to\infty\) and
	    using Lemma~\ref{qn_props}(4), \(q_{|x|}(t)\le w(x,t)\) for every \((x,t)\in\T\times[0,\infty)\).
    \end{proof}

    To show that \(q_k(t)\to0\) as \(k\to\infty\), for every \(t>0\),
    we need two preliminary lemmas: a Lyapunov-type function for
    \(\Delta_\beta\), and the comparison it buys us through the
    \(M\)-matrix structure of \(A_n\).

    \begin{lem}
	    \label{Phi_lyapunov}
	    Let \(\delta\coloneqq 1-2\beta>0\)
	    and \(a\ge a_0\coloneqq(\p+1-2\beta)/(1-\p)\). Then
	    \(\Phi(x)\coloneqq(|x|+a)\p^{|x|}\) satisfies
	    \(\Delta_\beta\Phi(x)\le-\delta\) for every \(x\in\T\).
    \end{lem}
    
    \begin{proof}
    Since \(\beta/\p=1-\beta\) and
    \((1-\beta)\p=\beta\), for \(|x|=k\ge1\) we have
    \[
        \Delta_\beta\Phi(x)
        =(1-\beta)(k-1+a)+\beta(k+1+a)-(k+a)
        =2\beta-1=-\delta.
    \]
    At the root,
    \[
        \Delta_\beta\Phi(\emptyset)
        =\p-a(1-\p)\le-\delta,
    \]
    because \(a\ge a_0=(\p+\delta)/(1-\p)\).
	\end{proof}

    \begin{lem}
	    \label{q_decae}
	    For every \(n\ge1\), \(0\le k\le n\), and \(t>0\),
	    \[
		    q_k^{(n)}(t)\ \le \frac{\Phi(k)}{\delta t}.
	    \]
	    Consequently \(q_k(t)\le\dfrac{\Phi(k)}{\delta t}\)
	    for every \(k\ge0\), \(t>0\).
    \end{lem}

    \begin{proof}
	    We also write \(\Phi\) for the vector
		\((\Phi(0),\dots,\Phi(n))\).
		By Lemma~\ref{Mmatrix_An}, the vector
		\(T^{(n)}\coloneqq(-A_n)^{-1}\mathbf1\)
		is well defined and satisfies
		\(A_nT^{(n)}=-\mathbf1\).
		We divide the proof into three steps.

	    \noindent\emph{Step 1.} We show \(T^{(n)}\le\Phi/\delta\).
	    By Lemma~\ref{Phi_lyapunov} applied at every
	    level \(0,\dots,n\), \(A_n\dfrac{\Phi}\delta\le-\mathbf1\)
	    (at \(k=n\) the truncated row is even more
	    negative than the full one, since we discard a non-negative term).
	    Then
	    \[
		    A_n\left(\frac{\Phi}\delta-T^{(n)}\right)\le0.
	    \]
	    By Lemma~\ref{Mmatrix_An}, \((-A_n)^{-1}\) is non-negative, so
	    \[
		    \frac{\Phi}\delta-T^{(n)}=(-A_n)^{-1}
		    \left(-A_n\left(\frac{\Phi}\delta-T^{(n)}\right)\right)\ge0.
	    \]

	    \noindent\emph{Step 2.} We show
	    \(T^{(n)}=\int_0^\infty q^{(n)}(s)\,ds\).
	    By Lemma~\ref{Mmatrix_An},
	    \(A_n\) is a stable matrix, so
	    \[
		    q^{(n)}(s)=e^{sA_n}\mathbf1\to0 \text{ as } s\to\infty.
	    \]
	    Integrating
	    \eqref{odeq} on \((0,\infty)\)
	    gives
	    \[
		    A_n\int_0^\infty q^{(n)}(s)\,ds=q^{(n)}(\infty)-q^{(n)}(0)=-\mathbf1,
	    \]
	    that is
	    \[
		    \int_0^\infty q^{(n)}(s)\,ds=T^{(n)}.
	    \]

	    \noindent\emph{Step 3.} By Lemma~\ref{qn_props}(3),
	    \(q^{(n)}\) is non-increasing in \(t\), so for every
	    \(t>0\),
	    \[
		    t\,q_k^{(n)}(t)=\int_0^tq_k^{(n)}(t)\,ds
		    \le\int_0^tq_k^{(n)}(s)\,ds\le\int_0^\infty q_k^{(n)}(s)\,ds
		    =T_k^{(n)}.
	    \]
	    Combining Steps~1 and~3,
	    \(q_k^{(n)}(t)\le\dfrac{T_k^{(n)}}t\le\dfrac{\Phi(k)}{\delta t}\),
	    for every
	    \(n\ge1\), \(0\le k\le n\), and \(t>0\).
	    Letting \(n\to\infty\)
	    gives
	    \(q_k(t)\le\dfrac{\Phi(k)}{\delta t}\),
	    for every \(k\ge0\) and \(t>0\).
    \end{proof}

    We are now in a position to show the uniform 
    (Dini-type) convergence of \(q\) to the boundary.

    \begin{co}\label{unif_dini_q}
	    For every \(0<t_0\le T<\infty\), \(q_k(t)\to0\)
	    as \(k\to\infty\), uniformly for \(t\in[t_0,T]\).
    \end{co}

\begin{proof}
    By Lemma~\ref{q_decae},
    \[
        \sup_{t\in[t_0,T]}q_k(t)
        \le\frac{\Phi(k)}{\delta t_0}
        =\frac{(k+a)\p^k}{\delta t_0}
        \xrightarrow[k\to\infty]{}0.
    \]
\end{proof}

    {
    We record the convergence in Corollary~\ref{unif_dini_q} in this stronger, locally uniform
    form because it is the natural statement for a function that is radial and non-increasing in
    \(k\) (Lemma~\ref{q_monotona_k}); only the weaker, pointwise-in-\(t\) convergence is actually
    needed below for property (3) of Definition~\ref{def:barrier}.
    }

    We finally define, for \((x,t)\in\T\times[0,\infty)\),
    \[
	    \phi(x,t)\coloneqq q_{|x|}(t),\qquad \phi_k(t)\coloneqq q_k(t).
    \]

    \begin{teo}
	    \label{barrier_exists}
	    \(\phi\) is a barrier, in the sense of Definition~\ref{def:barrier}.
    \end{teo}

    \begin{proof}
	    By Lemma~\ref{q_solves}, \(\phi=\mathcal K_\beta^1\phi\)
	    exactly on \(\T\times[0,\infty)\): property (1). By
	    Lemma~\ref{qn_props}(2), \(0\le\phi\le1\): property (2).
	    Finally, since \(\phi\) depends on \(x\) only
	    through \(|x|\), \(x\to y\in\partial\T\)
	    means precisely \(|x|\to\infty\) along the ray to \(y\); by
	    Corollary~\ref{unif_dini_q},
	    \(\phi_k(t)\to0\) as \(k\to\infty\),
	    uniformly for \(t\) in compact subsets of
	    \((0,\infty)\): property (3).
    \end{proof}

%%%%%%%%%%%%%%%%%%%%%%%%%%%%%%%%%%%%%%%%%%%%%%%%%%%%%%%%%%%%%%%%%%%%%%%%%%%%%%%%%%%%%%%%%%%%%%%
\subsubsection{Exponential decay in time}
%%%%%%%%%%%%%%%%%%%%%%%%%%%%%%%%%%%%%%%%%%%%%%%%%%%%%%%%%%%%%%%%%%%%%%%%%%%%%%%%%%%%%%%%%%%%%%%
    Let \(\lambda\in\mathcal A\), with associated function
    \(v\), \(0<c<v<C\) such that
    \begin{equation}\label{eqv}
	    \Delta_\beta v+\lambda v\le0 \text{ on }\T.
    \end{equation}
    By Remark \ref{remark_cte_niveles},  we can assume that \(v\)
    is constant on each level. Write \(v_k\) for the
    value of \(v\) at level \(k\).

    \begin{lem}
	    \label{q_leq_vlambda}
	    Then
	    \[
		    q_k(t)\ \le\ \frac{v_k}{c}\,e^{-\lambda t}
		    \qquad\text{for every }k\ge0,\ t\ge0.
	    \]
    \end{lem}

    \begin{proof}
	    Set \(\tilde v_k(t)\coloneqq v_ke^{-\lambda t}/c\).
	    By \eqref{eqv}
	    \[
		    \beta v_{k-1}+(1-\beta)v_{k+1}-v_k\
		    \le\ -\lambda\p^kv_k\ \ (k\ge1), \qquad
		    v_1-v_0\ \le\ -\lambda v_0\ \ (k=0),
	    \]
	    dividing by \(c\) and multiplying by \(\p^{-k}e^{-\lambda t}\)
	    (resp.\ \(e^{-\lambda t}\)) gives
	    \[
		    \p^{-k}\left(\beta\tilde v_{k-1}(t)
		    +(1-\beta)\tilde v_{k+1}(t)-\tilde v_k(t)\right)\ \le\
		    -\lambda\tilde v_k(t)=\dot{\tilde v}_k(t)\ \ (k\ge1), \qquad
		    \tilde v_1(t)-\tilde v_0(t)\ \le\ \dot{\tilde v}_0(t).
	    \]
	    That is, \(\tilde v\) is a supersolution of
	    \eqref{ode} at every level.

	    Fix \(n\ge1\) and set \(z_k(t)\coloneqq q_k^{(n)}(t)-\tilde v_k(t)\)
	    for \(0\le k\le n\); since
	    \(v_k>c\), \(z_k(0)=1-v_k/c<0\). For \(0\le k\le n-1\), \(q^{(n)}_k\)
	    satisfies exactly the
	    untruncated recursion (as in the proof of Lemma~\ref{q_solves}), so
	    subtracting the two inequalities
	    above,
	    \[
		    \dot z_k(t)\ \le\ \p^{-k}\left(\beta z_{k-1}(t)+(1-\beta)
		    z_{k+1}(t)-z_k(t)\right)\ (1\le k\le n-1),
		    \qquad \dot z_0(t)\le z_1(t)-z_0(t).
	    \]
	    At \(k=n\), \(q^{(n)}_n\) satisfies the truncated equation
	    \(\dot q^{(n)}_n=\p^{-n}(\beta q^{(n)}_{n-1}-q^{(n)}_n)\),
	    which is the
	    untruncated one with the
	    non-negative term \((1-\beta)q^{(n)}_{n+1}\) dropped; dropping the
	    corresponding non-negative term
	    \((1-\beta)\tilde v_{n+1}\ge0\) from the (already valid)
	    supersolution inequality for \(\tilde v_n\)
	    only weakens it, so
	    \(\p^{-n}(\beta\tilde v_{n-1}-\tilde v_n)\le\dot{\tilde v}_n\)
	    still holds, and
	    subtracting gives \(\dot z_n(t)\le\p^{-n}(\beta z_{n-1}(t)-z_n(t))\)
	    as well.

	    In every case \(\dot z(t)\le A_nz(t)\) entrywise, with \(z(0)<0\)
	    entrywise, where \(A_n\) is exactly
	    the matrix of Lemma~\ref{Mmatrix_An}.
	    Writing \(p(t)\coloneqq A_nz(t)-\dot z(t)\ge0\), Duhamel gives
	    \[
		    z(t)\ =\ e^{tA_n}z(0)\ -\ \int_0^te^{(t-s)A_n}p(s)\,ds.
	    \]
	    By Lemma~\ref{qn_props}(1), \(e^{tA_n}\ge0\) entrywise,
	    so both terms on the right are non-positive:
	    hence \(z(t)\le0\), that is \(q^{(n)}_k(t)\le\tilde v_k(t)\)
	    for \(0\le k\le n\). Letting
	    \(n\to\infty\) and using Lemma~\ref{qn_props}(4)
	    gives \(q_k(t)\le\tilde v_k(t)\) for every \(k\ge0\),
	    \(t\ge0\).
    \end{proof}

%%%%%%%%%%%%%%%%%%%%%%%%%%%%%%%%%%%%%%%%%%%%%%%%%%%%%%%%%%%%%%%%%%%%%%%%%%%%%%%%%%%%%%%%%%%%%%%
\subsection{Uniform boundary behavior for data vanishing at the boundary}
\label{subsec:phi_exist_unique}
%%%%%%%%%%%%%%%%%%%%%%%%%%%%%%%%%%%%%%%%%%%%%%%%%%%%%%%%%%%%%%%%%%%%%%%%%%%%%%%%%%%%%%%%%%%%%%%

\subsubsection{A general Lyapunov majorant}

    We now show that a non-negative function on \(\T\) with \(\Delta_\beta w\le0\) (that is, superharmonic
    for \(\Delta_\beta\)) produces a time-independent supersolution of its \emph{own} datum problem, with no
    exponential-in-time factor needed. \(\Phi\) of Lemma~\ref{Phi_lyapunov} is one instance, but far from the
    only one; below we build, out of \(f\) itself, an explicit superharmonic majorant that vanishes at
    \(\partial\T\) exactly as slowly as \(f\) does, which is what lets Theorem~\ref{teo_exist_unic_bdry} below
    cover \emph{every} datum vanishing at the boundary, with no rate assumption at all.

    \begin{lem}
	    \label{Phi_supersolution}
	    Let \(w\colon\T\to\R\), \(w\ge0\), satisfy \(\Delta_\beta w\le0\) on \(\T\). Then the time-independent function
	    \(W(x,t)\coloneqq w(x)\) satisfies \(W\ge\mathcal K_\beta^wW\) on \(\T\times[0,\infty)\).
    \end{lem}

    \begin{proof}
	    By the Duhamel formula \eqref{Kbeta} defining \(\mathcal K_\beta^{(\cdot)}\),
	    \begin{align*}
		    \mathcal K_\beta^ww(x,t)
		    &=e^{-tp^{-|x|}}w(x)+\int_0^te^{(s-t)p^{-|x|}}\left(\Delta_\beta w(x)+p^{-|x|}w(x)\right)ds\\
		    &\le e^{-tp^{-|x|}}w(x)+p^{-|x|}w(x)\int_0^te^{(s-t)p^{-|x|}}ds
		    \qquad(\text{since }\Delta_\beta w(x)\le0)\\
		    &= e^{-tp^{-|x|}}w(x)+w(x)\left(1-e^{-tp^{-|x|}}\right)\\
		    &= w(x).
	    \end{align*}
    \end{proof}

    Superharmonicity of a radial function is a purely discrete condition (for \(M\colon\N_0\to[0,\infty)\),
    writing \(w_k\coloneqq w(x)\) for \(|x|=k\), it reads \(\beta w_{k-1}+(1-\beta)w_{k+1}\le w_k\) for \(k\ge1\))
    and, unlike what the specific choice \(\Phi\) might suggest, it places essentially no restriction on how
    \emph{slowly} \(w\) may decay to \(0\); it only forbids decaying faster than the geometric rate \(\p\). The
    next lemma turns this observation into an explicit construction: given any non-increasing envelope
    \(M(k)\to0\), however slowly, there is a superharmonic majorant of \(M\) that still vanishes at infinity.

    \begin{lem}
	    \label{mayorante_general}
	    Let \(M\colon\N_0\to[0,\infty)\) be non-increasing with
	    \[
		    M(k)\to0 \quad\text{ as }k\to\infty.
	    \]
	    Define
	    \[
		    \tilde w_k\ \coloneqq\ (1-\p)\sum_{j=0}^{k}M(j)\,\p^{\,k-j}, \qquad k\ge0.
	    \]
	    Then \(\tilde w_k\ge(1-\p)M(k)\) for every \(k\ge0\), \(\tilde w_k\to0\) as \(k\to\infty\), and
	    \[
		    \beta\tilde w_{k-1}+(1-\beta)\tilde w_{k+1}-\tilde w_k\ =\ (1-\beta)(1-\p)\left(M(k+1)-M(k)\right)\ \le\ 0
		    \qquad\text{for every }k\ge1.
	    \]
    \end{lem}

    \begin{proof}
	    Directly from the definition,
	    \[
		    \tilde w_k=\p\,\tilde w_{k-1}+(1-\p)M(k)
	    \]
	    for \(k\ge1\). Since \(M\) is non-increasing, every term in the defining sum satisfies
	    \[
		    M(j)\ge M(k) \quad\text{for }j\le k,
	    \]
	    and the weights \((1-\p)\p^{k-j}\) sum to \(1-\p^{k+1}\ge1-\p\), so
	    \[
		    \tilde w_k\ge(1-\p)M(k).
	    \]

	    Now fix \(\varepsilon>0\) and split the sum at a level \(J\) (to be chosen),
	    \[
		    \tilde w_k=(1-\p)\sum_{j=0}^{J}M(j)\,\p^{k-j}+(1-\p)\sum_{j=J+1}^{k}M(j)\,\p^{k-j}, \qquad k>J.
	    \]
	    For the first sum, \(M(j)\le M(0)\) for every \(j\), so
	    \[
		    (1-\p)\sum_{j=0}^{J}M(j)\,\p^{k-j}\ \le\ M(0)\,\p^{k-J}(1-\p)\sum_{j=0}^{J}\p^{J-j}\ \le\ M(0)\,\p^{k-J},
	    \]
	    For the second sum, \(M\) is non-increasing, so \(M(j)\le M(J+1)\) for every \(j>J\), and
	    \[
		    (1-\p)\sum_{j=J+1}^{k}M(j)\,\p^{k-j}\ \le\ M(J+1)\,(1-\p)\sum_{i=0}^{k-J-1}\p^{i}\ \le\ M(J+1),
	    \]
	    using \((1-\p)\sum_{i\ge0}\p^i=1\). Hence
	    \[
		    \tilde w_k\le M(0)\,\p^{k-J}+M(J+1) \quad \text{for every  }k>J.
	    \]
	    Then we choose \(J\) so that \(M(J+1)<\varepsilon/2\) (possible since \(M(j)\to0\)), then choose
	    \(k\) large enough that \(M(0)\,\p^{k-J}<\varepsilon/2\) (possible since \(\p<1\)); this gives
	    \[
		    \tilde w_k<\varepsilon \quad\text{for every such }k,
	    \]
	    and, since \(\varepsilon>0\) was arbitrary, \(\tilde w_k\to0\) as
	    \(k\to\infty\).

	    For the identity, use \(\p(1-\beta)=\beta\) (that is, \((1-\beta)\p=\beta\), from the definition of \(\p\))
	    twice, together with the recursion \(\tilde w_k=\p\tilde w_{k-1}+(1-\p)M(k)\) applied at \(k\) and at
	    \(k+1\):
	    \begin{align*}
		    \beta\tilde w_{k-1}&+(1-\beta)\tilde w_{k+1}-\tilde w_k
		    =\beta\tilde w_{k-1}+(1-\beta)\left[\p\tilde w_k+(1-\p)M(k+1)\right]-\tilde w_k\\
		    &=\beta\tilde w_{k-1}-(1-\beta)\tilde w_k+(1-\beta)(1-\p)M(k+1)\\
		    %\qquad(\text{using }(1-\beta)\p=\beta)\\
		    &=\beta\tilde w_{k-1}-(1-\beta)\left[\p\tilde w_{k-1}+(1-\p)M(k)\right]+(1-\beta)(1-\p)M(k+1)\\
		    &=\left[\beta-(1-\beta)\p\right]\tilde w_{k-1}+(1-\beta)(1-\p)\left(M(k+1)-M(k)\right)\\
		    &=(1-\beta)(1-\p)\left(M(k+1)-M(k)\right),
	    \end{align*}
	    %using $(1-\beta)\p=\beta$ once more in the last step,
	    then the coefficient of \(\tilde w_{k-1}\) vanishes
	    exactly. Since \(M\) is non-increasing,
	    \[
		    M(k+1)-M(k)\le0.
	    \]
    \end{proof}

    Since \(\tilde w\) need not be superharmonic at the root (the identity above only covers \(k\ge1\)), we
    patch it there using \(\Phi\), which already is, by a fixed margin:

    \begin{co}
	    \label{mayorante_patched}
	    With \(M,\tilde w\) as in Lemma~\ref{mayorante_general}, there is \(\varepsilon>0\), depending only on
	    \(M(0),M(1)\) and the constants \(a,\delta\) of Lemma~\ref{Phi_lyapunov}, such that
	    \(w\coloneqq\tilde w+\varepsilon\Phi\) satisfies \(\Delta_\beta w\le0\) on all of \(\T\), \(w\ge(1-\p)M\), and
	    \(w(x)\to0\) as \(|x|\to\infty\).
    \end{co}

    \begin{proof} Since	\(\Delta_\beta\) is linear, so for \(k\ge1\),
	    \[
		    \Delta_\beta w\le0+\varepsilon\Delta_\beta\Phi\le0
	    \]
	    by Lemmas \ref{mayorante_general} and \ref{Phi_lyapunov}, for \emph{any} \(\varepsilon\ge0\).
	    At the root,
	    \[
		    \Delta_\beta\tilde w(\emptyset)=\tilde w_1-\tilde w_0
	    \]
	    is a fixed real number determined by \(M(0),M(1)\) alone, while
	    \[
		    \Delta_\beta\Phi(\emptyset)\le-\delta<0
	    \]
	    by Lemma~\ref{Phi_lyapunov}. Then, taking
	    \(\varepsilon\ge(\tilde w_1-\tilde w_0)^+/\delta,\) we have
	    \[
		    \Delta_\beta w(\emptyset)=(\tilde w_1-\tilde w_0)+\varepsilon\Delta_\beta\Phi(\emptyset)\le0.
	    \]  Domination and vanishing at
	    infinity are immediate, since \(\tilde w\ge(1-\p)M\), \(\tilde w\to0\) (Lemma~\ref{mayorante_general}) and
	    \(\Phi\to0\) (Lemma~\ref{Phi_lyapunov}).
    \end{proof}

    \subsubsection{Existence and uniqueness}

\begin{teo}
    \label{teo_exist_unic_bdry}
    {Let \(\beta\in(0,\tfrac12)\) and \(f\in L^\infty(\T)\) satisfy condition \eqref{cond:L}.}
    Then the solution \(u_f\) given by
    Theorem~\ref{teo_exist_unic_v2} satisfies
    \[
        \sup_{\substack{|x|\ge N\\t\ge0}}
        |u_f(x,t)|\xrightarrow[N\to\infty]{}0.
    \]
\end{teo}

    \begin{proof}
	    \textbf{Existence.}
	    Let
	    \[
		    M(k)\coloneqq\sup_{|x|\ge k}|f(x)|.
	    \]
	    By hypothesis \(M\) is non-increasing (a supremum over a shrinking set) and
	    \[
		    M(k)\to0.
	    \]
	    Let \(w\coloneqq\tilde w+\varepsilon\Phi\) be the majorant of Corollary~\ref{mayorante_patched}, so
	    \[
		    w\ge0, \quad \Delta_\beta w\le0, \quad w(x)\to0  \text{ as } |x|\to\infty,
	    \]
	    and
	    \[
		    (1-\p)|f(x)|\le(1-\p)M(|x|)\le w(x) \quad \text{ for every }x\in\T.
	    \]
	    By Lemma~\ref{Phi_supersolution},
	    \(w\ge\mathcal K_\beta^ww\);
	    since \(|f|\le w/(1-\p)\), monotonicity of \(\mathcal K_\beta^{(\cdot)}\) in the
	    datum  gives
	    \[
		    \mathcal K_\beta^f(w/(1-\p))\le\mathcal K_\beta^{w/(1-\p)}(w/(1-\p))\le w/(1-\p),
	    \]
	    that is,
	    \[
		    w_1\coloneqq\frac{w}{1-\p}\ \ge\ \mathcal K_\beta^f w_1, \qquad
		    w_2\coloneqq-\frac{w}{1-\p}\ \le\ \mathcal K_\beta^f w_2.
	    \]
	    By Theorem~\ref{existencia_sin_cp}, there is an exact solution
	    \[
		    \widetilde u_f=\mathcal K_\beta^f \widetilde u_f
	    \]
	    with
	    \[
		    |\widetilde u_f(x,t)|\le \frac{w(x)}{1-\p} \quad \text{ for every }(x,t)\in\T\times[0,\infty);
	    \]
	    since this bound does not depend on \(t\) and \(w(x)\to0\) as \(|x|\to\infty\),
	    \[
		    \sup_{|x|\ge N,\,t\ge0}|\widetilde u_f(x,t)|\to0
	    \]
	    as \(N\to\infty\), so \(\widetilde u_f\) vanishes at \(\partial\T\) uniformly on \([0,T]\)
	    for every \(T<\infty\) and solves \eqref{ec.evolucion} with datum \(f\).

	   \noindent \textbf{Uniqueness.} 
		The function \(\widetilde u_f\) is a globally bounded
		solution of \eqref{ec.evolucion} with initial datum \(f\).
		By Lemma~\ref{unicidad_general}, it coincides with the
		solution \(u_f\) of Theorem~\ref{teo_exist_unic_v2}.
		The preceding estimate therefore proves the assertion.    
    \end{proof}

    \begin{re}
	    Theorem~\ref{teo_exist_unic_bdry} also subsumes the
	    qualitative part of Theorem~\ref{teo_evo_2} below: the principal eigenfunction \(v\) is
	    radial and non-increasing in the level (Theorem~\ref{primer_teorema}), so
	    \[
		    \sup_{|x|\ge N}v(x)=v_N\xrightarrow[N\to\infty]{}\ 0,
	    \]
	    and \(f\) with \(|f|\le Kv\) satisfies the hypothesis above. What Theorem \ref{teo_evo_2} adds,
	    and Theorem \ref{teo_exist_unic_bdry} does not, is the \emph{rate} \(e^{-\lambda_1(\beta)t}\) 
	    of decay in time,
	    a genuinely quantitative statement that needs \(v\) to be an
	    exact eigenfunction, not just a superharmonic majorant, and is proved directly below.
    \end{re}

%%%%%%%%%%%%%%%%%%%%%%%%%%%%%%%%%%%%%%%%%%%%%%%%%%%%%%%%%%%%%%%%%%%%%%%%%%%%%%%%%%%%%%%%%%%
\section{Asymptotic Decay}
\label{sec:decay}
    Having established existence and uniqueness in Section~\ref{eq}, we now turn to the rate at
    which \(u_f\) decays. Theorem~\ref{teo_evo_1} gives an explicit exponential rate governed by the
    spectral parameter of a chosen barrier, and Theorem~\ref{teo_evo_2} compares \(u_f\) directly against
    an explicit supersolution.

    \begin{teo}
	    \label{teo_evo_1}
	    Let \(\beta\in(0,\tfrac12)\), \(f\in L^{\infty}(\T)\), and \(u_f\) the solution of
	    \eqref{ec.evolucion} given by Theorem~\ref{teo_exist_unic_v2}.
	    Then for any \(\lambda\in\A\), with associated barrier \(v\), \(0<c<v<C\),
	    \[
		    |u_f(x,t)|\le \frac{\|f\|_\infty\,C}{c}\, e^{-\lambda t} \quad
		    \forall (x,t)\in\T\times(0,\infty).
	    \]
    \end{teo}

    \begin{proof}
	    By Lemma~\ref{q_leq_vlambda},
	    \(\phi(x,t)=q_{|x|}(t)\le v(x)e^{-\lambda t}/c\le(C/c)e^{-\lambda t}\). By
	    Lemma~\ref{sandwich_f},
	    \[
		    |u_f(x,t)|\ \le\ \|f\|_\infty\,\phi(x,t)\ \le\ \frac{\|f\|_\infty\,C}{c}\,e^{-\lambda t}.
	    \]
    \end{proof}

    \begin{teo}
	    \label{teo_evo_2}
	    Let \(\beta\in(0,\tfrac12)\), and \(f\in L^{\infty}(\T)\) such that there exists \(K>0\) with
	    \[
		    |f(x)|\le K\,v(x) \quad \forall x\in \T,
	    \]
where \(v\) is the positive radial eigenfunction constructed
in Theorem~\ref{existencia-autofuncion}, normalized by
\(v(\emptyset)=1\).
		 Then the unique globally bounded solution \(u_f\)
		of \eqref{ec.evolucion} satisfies
	    \[
		    |u_f(x,t)|\le K\,e^{-\lambda_1(\beta)t} \quad \forall (x,t)\in\T\times(0,\infty).
	    \]
    \end{teo}

    \begin{proof}
	    Since
	    \[
		    \Delta_\beta v=-\lambda_1(\beta)v
	    \] exactly,
	    \[
		    W(x,t)\coloneqq Kv(x)e^{-\lambda_1(\beta)t}
	    \]
	    solves
	    \[
		    W=\mathcal K_\beta^{Kv}W
	    \] exactly. As \(f\le Kv\), monotonicity of
	    \(\mathcal K_\beta^{(\cdot)}\) in the datum gives
	    \[
		    \mathcal K_\beta^fW\le\mathcal K_\beta^{Kv}W=W,
	    \]
	    so \(w_1\coloneqq W\) satisfies \(w_1\ge\mathcal K_\beta^fw_1\).
	    Applying the same computation to
	    \(-f\le Kv\)
	    gives
	    \[
		    w_2\coloneqq-W\le\mathcal K_\beta^fw_2.
	    \]
	    By Theorem~\ref{existencia_sin_cp}, there is an exact
	    solution \(u_f=\mathcal K_\beta^fu_f\) with \(-W\le u_f\le W\), that is,
	    \[
		    |u_f(x,t)|\ \le\ K\,v(x)\,e^{-\lambda_1(\beta)t} \qquad\text{for every }(x,t)\in\T\times[0,\infty).
	    \]
	    Since \(v\) is constant on levels and non-increasing in the level, with \(v(\emptyset)=1\)
	    (Theorem~\ref{primer_teorema}), \(v(x)\le v(\emptyset)=1\) for every \(x\in\T\), so
	    \[
		    |u_f(x,t)|\ \le\ K\,e^{-\lambda_1(\beta)t} \qquad\forall (x,t)\in\T\times(0,\infty).
	    \]

	    Finally, since \(v(x)\to0\) as \(x\to y\in\partial\T\) (Theorem~\ref{existencia-autofuncion}) and
	    \(e^{-\lambda_1(\beta)t}\le1\) for every \(t\ge0\),
	    \[
		    \sup_{|x|\ge N,\,t\ge0}|u_f(x,t)|\ \le\ K\sup_{|x|\ge N}v(x)\ \xrightarrow[N\to\infty]{}\ 0,
	    \]
	    so \(u_f\) vanishes at \(\partial\T\) uniformly on all of \([0,\infty)\).
	    
	    By Lemma~\ref{unicidad_general}, this solution coincides with the unique globally bounded solution of \eqref{ec.evolucion} with initial datum \(f\).

    \end{proof}

%%%%%%%%%%%%%%%%%%%%%%%%%%%%%%%%%%%%%%%%%%%%%%%%%%%%%%%%%%%%%%%%%%%%%

\section*{Acknowledgments}

This work was partially supported by the grants 21-MATH-04-NoLoCE and 23-MATH-08-MiLNE from Math Amsud program.
Del Pezzo and Frevenza are supported by ANII FCE-3-2024-1-181302 (Uruguay).
Frevenza was funded by the Fondo Vaz Ferreira grant FVF-061-2021 (Uruguay).

%%%%%%%%%%%%%%%%%%%%%%%%%%%%%%%%%%%%%%%%%%%%%%%%%%%%%%%%%%%%%%%%%%%%%%%%%%%%%%%%%%%%%%%%%%%%%%%
\bibliographystyle{amsplain}

%%%%%%%%%%%%%%%%%%%%%%%%%%%%%%%%%%%%%%%%%%%%%%%%%%%%%%%%%%%%%%%%%%%%%%%%%%%%%%%%%%%%%%%%%%%%%%%

\end{document}